\PassOptionsToPackage{table}{xcolor}

\documentclass[10pt, a4paper, reqno]{amsart}

\usepackage[left=1.15in,right=1.15in,top=1.15in,bottom=1.15in]{geometry}

\usepackage{amsmath,amsthm}

\newtheorem{lemma}{Lemma}
\newtheorem{proposition}[lemma]{Proposition}
\newtheorem{theorem}[lemma]{Theorem}
\numberwithin{lemma}{section}
\newtheorem{corollary}[lemma]{Corollary}

\theoremstyle{definition}
\newtheorem{definition}[lemma]{Definition}

\theoremstyle{definition}
\newtheorem{remark}[lemma]{Remark}

\theoremstyle{definition}

\theoremstyle{definition}
\newtheorem{example}[lemma]{Example}

\theoremstyle{theorem}

\theoremstyle{theorem}

\theoremstyle{theorem}

\author{Alexander Bastounis}
\address{University of Leicester,
		UK}
\email{ajb177@leicester.ac.uk}

\author{Felipe Cucker}
\address{City University of Hong Kong, 
		HONG KONG}
\email{macucker@gmail.com}

\author{Anders C. Hansen}
\address{University of Cambridge, UK}
\email{ach70@cam.ac.uk}

\usepackage{amssymb}

\numberwithin{equation}{section}
\usepackage{bbm}
\usepackage{enumitem}
\usepackage{cases}
\usepackage{calc}
\usepackage{ifthen,tabularx,graphicx,multirow}
\graphicspath{ %
	{Figures/} %
	{Figures/Zoom} %
}
\usepackage{rotating}
\usepackage{hyperref}
\usepackage{tikz}

\usepackage{times}

\usetikzlibrary{tikzmark}
\usepackage{ etoolbox}
\usepackage{tcolorbox}

\usepackage[algo2e]{algorithm2e}
\usepackage{cite}
\newcolumntype{Y}{>{\centering\arraybackslash}X}
\newcommand{\triple}[1]{[\![(#1)]\!]}

\newcommand{\trip}[1]{[\![#1]\!]}

\theoremstyle{definition}

\newcount\colveccount
\newcommand*\colvec[1]{
	\global\colveccount#1
	\begin{pmatrix}
		\colvecnext
	}
	\def\colvecnext#1{
		#1
		\global\advance\colveccount-1
		\ifnum\colveccount>0
		\\
		\expandafter\colvecnext
		\else
	\end{pmatrix}
	\fi
}

\makeatletter
\newcounter{framedeqn}
\AtBeginEnvironment{subequations}%
{\let\c@equation\c@framedeqn}
\makeatother

\newcommand{\real}{\mathbb{R}}
\newcommand{\nat}{\mathbb{N}}
\newcommand{\indic}{\mathbbm{1}}
\newcommand{\sg}{\mathsf{sgn}}

\newcommand{\argmin}{\mathop\mathrm{argmin}}

\newcommand{\supp}{\mathsf{supp}}

\newcommand{\cond}[1]{\mathrm{cond}(#1)}

\newcommand{\rational}{\mathbb{Q}}

\usepackage{amsfonts}
\usepackage{mathrsfs}
\usepackage{tikz}
\usepackage{float}
\usepackage{csquotes}

\newcommand{\xul}{\xi_{\mathrm{UL}}}
\newcommand{\condfsul}{\mathscr{C}_{\mathrm{UL}}}
\newcommand{\Sigmul}{\mathsf{\Sigma}_{\mathrm{UL}}}
\newcommand{\stabsupp}{\mathsf{stsp}}
\newcommand{\Sol}{\mathsf{Sol}}
\newcommand{\mUL}{\mathsf{Sol}^{\mathsf{UL}}}

\newcommand{\rmd}{\,\mathrm{d}}
\newcommand{\bfc}{\mathbf{c}}
\renewcommand{\bar}{\overline}

\allowdisplaybreaks
\newcommand{\transp}{^{\mathrm{T}}}

\newcommand{\Id}{\/\mathrm{I}}
\newcommand{\eproof}{\hfill\qed}

\newcommand{\probab}{\mathbb{P}}

\def\Oh{\mathcal O}
\def\mcN{\mathcal N}

\newcommand{\nMatSing}[2]{\|#1\|_{#2}}

\newcommand{\nmax}[1]{\|{#1}\|_{\max}}

\newcommand{\nTwoTwo}[1]{\nMatSing{#1}{2}}

\newcommand{\nTrMax}[1]{\triple{#1}_{\max}}
\newcommand{\nTrTwo}[1]{\triple{#1}_2}
\newcommand{\nTrOneStar}[1]{\triple{#1}_{\mathrm{S}}}

\def\Oh{\mathcal O}
\def\mcN{\mathcal N}
\def\mcU{\mathcal U}
\def\sg{\mathop{\mathsf{sgn}}}
\renewcommand{\xul}{\stabsupp}
\def\bfC{\mathbf{C}}
\def\bfS{\mathbf{S}}

\title[On the effects of random data on condition in statistics and optimisation]{When can you trust feature selection? -- II: On the effects of random data on condition in statistics and optimisation}

\newcommand{\mULMS}{\Xi^{\text{ms}}}

\newcommand{\rme}{\mathrm{e}}
\DeclareMathOperator{\sech}{sech}

\DeclareMathOperator{\erf}{erf}
\begin{document}
	\maketitle		

\begin{abstract}
In Part I, we defined a LASSO condition number and developed an algorithm -- for computing support sets (feature selection) of the LASSO minimisation problem -- that runs in polynomial time in the number of variables and the logarithm of the condition number.  The algorithm is trustworthy in the sense that if the condition number is infinite, the algorithm will run forever and never produce an incorrect output. In this Part II article, we demonstrate how finite precision algorithms (for example algorithms running floating point arithmetic) will fail on open sets when the condition number is large -- but still finite. This augments Part I's result: If an algorithm takes inputs from an open set that includes at least one point with an infinite condition number, it fails to compute the correct support set for all inputs within that set. Hence, for any finite precision algorithm working on open sets for the LASSO problem with random inputs, our LASSO condition number -- as a random variable -- will estimate the probability of success/failure of the algorithm. We show that a finite precision version of our algorithm works on traditional Gaussian data for LASSO with high probability. The algorithm is trustworthy, specifically, in the random cases where the algorithm fails, it will not produce an output. Finally, we demonstrate classical random ensembles for which the condition number will be large with high probability, and hence where any finite precision algorithm on open sets will fail.  We show numerically how commercial software fails on these cases.  
\end{abstract}

\section{Introduction}\label{sec:intro}

This article is the continuation of~\cite{BCH1}, which in the sequel we 
will refer to as Part~I. Both here and in Part~I, the \emph{unconstrained LASSO feature selection problem}
\cite{LassoStart, SLSBook} is the main focus. Specifically, we are interested in computing, for fixed $\lambda\in \rational$, $\lambda > 0$, an element in 
\begin{equation}\label{eq:LassoComp}
   \Xi(y,A) = \lbrace \supp(x) \, \vert \,
   x \in \argmin_{\hat{x} \in \real^{N}}\|A\hat{x}-y\|^2_2 
   + \lambda \|\hat{x}\|_1\rbrace,
\end{equation}
where $(y,A) \in\real^m\times \real^{m\times N}$. The rationale is as follows: given the many AI-based
algorithms
in the computational sciences, with the potential for hallucinations and
non-robustness \cite{antun2020instabilities, jin17,
mccann2017convolutional, Choi,DezFaFr-17, hammernik2018learning, heaven2019deep, Anders2, Choi, Choi2}, the question of
trustworthiness of algorithms is now becoming a crucial topic. 
For example, the European
Commission~\cite{EU_Commission_2020} has been particularly vocal about
its demand for trust in algorithms.
However, with this new focus on
trust in algorithms comes an important question: Which of the
classical (non-AI-based) approaches are trustworthy, such as LASSO feature selection?  

Part I first defines a LASSO condition number $\condfsul(b,U)$ (see Definition \ref{def:cond}) for any pair $(b,U)\in\real^m\times \real^{m\times N}$, and then provides the following Theorem (Theorem~1.2
there\footnote{In all what follows, for simplicity, we will use a prefix
	`I.' in the references to objects in Part~I not contained here. Thus, for instance,
	Theorem~1.2 or equation~(7.19) there become Theorem~I.1.2 and
	equation~(I.7.19) here.}) below. The model of computation is that any algorithm reads variable-precision 
approximations of the input $(b,U)\in\real^m\times \real^{m\times N}$. Most importantly, the set of variable precision algorithms contains the set of finite precision algorithms, which is a typical way of modelling algorithms using floating point arithmetic.

\begin{remark}[Model of computation -- Inexact input]\label{rem:1}
In practice, when trying to compute an element of $\Xi(y,A)$
in~\eqref{eq:LassoComp}, we must assume that the $A$ and $y$ are given
inexactly. This is because either we have: (1) an irrational input; or
(2) the input is rational (for example $1/3$), but our computer
expresses numbers in a certain base (typically base-2); (3) the
computer uses floating-point arithmetic for which -- in many cases --
the common backward-error analysis (popularized by
Wilkinson~\cite{Wilkinson63}) translates the accumulation of round-off
in a computation into a single-perturbation of the input data. Hence, we assume that algorithms access the input to
whatever finite precision desired and that all computational operations
are done exactly.
\end{remark}

\begin{theorem}[Main theorem of Part I]\label{thm:fsulcomp}
	Consider the condition number $\condfsul(b,U)$ defined
	in~\eqref{eq:CondDefinition}.
	\begin{itemize}
		\item[(1)]  We exhibit an algorithm $\Gamma$ which, for any input pair 
		$(b,U)\in\real^m\times\real^{m\times N}$, reads variable-precision 
		approximations of $(b,U)$. If $\condfsul(b,U) < \infty$
		then the algorithm halts and returns a correct value in $\Xi(b,U)$.
		The cost of
		this computation is
		\begin{equation*}
			\Oh\left\{N^{3}\left[\log_2\left(N^2 \triple{b,U}_{\max}^2
			\condfsul(b,U)\right)\right]^2\right\}.
		\end{equation*}
		\label{item:fsulcomppos}
		If, instead, $\condfsul(b,U)=\infty$ then the algorithm runs forever.
		
		\item[(2)]
		The condition number $\condfsul(b,U)$ can be estimated in the
		following sense: There exists an algorithm that provides an upper bound
		on $\condfsul(b,U)$, when it is finite, and runs forever when
		$\condfsul(b,U) = \infty$.
		
		\item[(3)]
		If $\Omega\subseteq \real^m\times\real^{m\times N}$ is an open set and
		there is a $(b,U)\in\Omega$ with $\condfsul(b,U) = \infty$ then there
		is no algorithm that, for all input $(y,A)\in\Omega$, computes an element
		of $\Xi(y,A)$ given approximations to $(y,A)\in \Omega$.  Moreover, for any
		randomised algorithm $\Gamma^{\mathrm{ran}}$ that always halts and any $p>1/2$, there exists
		$(y,A)\in\Omega$ and an approximate representation $(\tilde{y},\tilde{A})$ (for a precise statement see \S 9 in Part I \cite{BCH1})
		of $(y,A)$ so that $\Gamma^{\mathrm{ran}}(\tilde{y},\tilde{A}) \notin \Xi(y,A)$ with probability
		at least $p$.
		
		 If $(b,U)\in\Omega$ is computable, then the failure point $(y,A)\in\Omega$ above can be made computable.  
	\end{itemize}
	\label{item:fsulcompneg}
\end{theorem}

\begin{remark}[Trustworthiness of algorithms -- No wrong outputs]\label{rem:trust}
By `trustworthy algorithm' for a computational problem, we mean the following. If the computational problem takes only discrete values (as is the case when computing support sets of minimisers of optimisation problems), a trustworthy algorithm will always produce a correct answer -- if it halts. 
\end{remark}

Note that (1) implies -- in view of the question on trustworthiness -- that our algorithm will never output a wrong solution, and thus if it halts, the output is always trustworthy. However, there are inputs for which it will not produce an answer. In view of (3) this is optimal in terms of existence of algorithms on open sets: Every algorithm will fail on some inputs, although not necessarily on inputs $(b,U)$ where $\condfsul(b,U) = \infty$, which is where our algorithm fails. 

\subsection{The LASSO problem with random inputs}

To motivate our main results in this paper -- Part II -- we begin by asking the basic question:
\vspace{1mm}
\begin{displayquote}
\normalsize
{\it Can commercial software for the LASSO problem be trusted on random data? If not, why? Is there a link to our LASSO condition number? And how can a lack of trust be mitigated?}
\end{displayquote}
\vspace{1mm}
To illustrate the motivation behind this question we consider the following example, using a variety of different probability distributions on the input data. 

\begin{figure}
	\begin{center}
	\begin{tabular}{c c}
		\includegraphics[scale = 0.8]{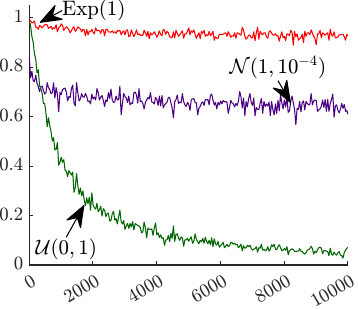} & \includegraphics[scale = 0.8]{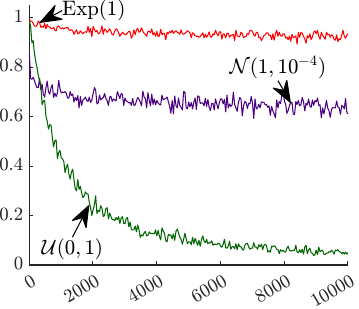} \\
		\!\!\includegraphics[scale =0.8]{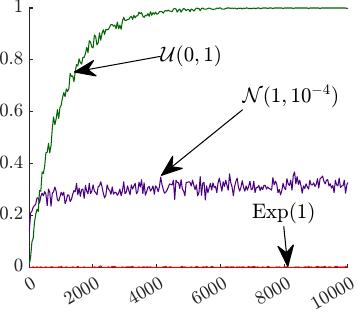} &\includegraphics[scale = 0.8]{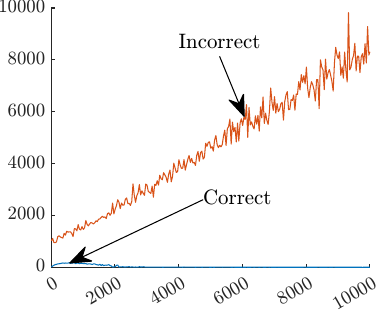}
	\end{tabular} 
	\end{center}
\caption{
\emph{({\bf $\mathbb{P}(\text{LASSO has a unique minimiser}) = 1$ and $\mathbb{P}(\condfsul(b,U) < \infty) = 1$, yet standard algorithms fail to compute the support set})}. Testing MATLAB's {\tt lasso} on random iid inputs $(b,U)\in\real^1\times \real^{1\times N}$ -- according to the distributions $\mathcal{U}(a,b)$ (uniform), $\mathrm{Exp}(\nu)$ (exponential) and $\mathcal{N}(\mu, \sigma^2)$ (normal). The task is to compute the support set of a LASSO minimiser i.e. an element in $\Xi(b,U)$ from \eqref{eq:LassoComp} with $\lambda = 10^{-2}$. All figures: the horizontal axis represents the dimension $N$. Top figures: the vertical axis represents the success rate $\frac{\# \text{ of successes }}{\# \text{ of trials }}$ with threshold value -- see the text accompanying \eqref{eq:TheSolSet} -- set to $10^{-3}$ and $10^{-12}$ for the left and right figures respectively. Bottom left figure: the proportion of trials (threshold = $10^{-12}$) for each dimension which had a condition $\condfsul(b,U)$ above $1,000$, for each distribution. Bottom right figure: the vertical axis is the median condition $\condfsul(b,U)$ (across all distributions and trials) for that dimension for the cases where MATLAB was correct and MATLAB was incorrect (threshold = $10^{-12}$).
Note that for all the distributions considered $\mathbb{P}(\text{LASSO has a unique minimiser}) = 1$ and $\mathbb{P}(\condfsul(b,U) < \infty) = 1$ for all $N$.}
\label{fig:ExampleNDRandomResult}
\end{figure}

\begin{example}[{\bf Testing algorithms for LASSO with random inputs}]\label{Example:NDRandom}
We set $m=1$, $b=1$, and $\lambda = 10^{-2}$ and generate each entry of a matrix $U \in \real^{1 \times N}$ iid from three distributions: the exponential distribution with parameter $1$, the normal distribution with mean $1$ variance $10^{-4}$, and the uniform distribution on $(0,1)$.  This purposefully simplified situation is considered because it is easy to use Lemma~\ref{lemma:1DZeroSupportRequirement}
to compute $\Xi(b,U)$ as defined in \eqref{eq:LassoComp} (this is a singleton with probability~1). We compare this answer (the ground truth) with the following procedure: we use Matlab's lasso routine to attempt to compute an element $x$ in
\begin{equation}\label{eq:TheSolSet}
\mUL(b,U) := \argmin\limits_{\hat{x}\in \real^{N}}
\|U\hat{x}-b\|^2_2 + \lambda \|\hat{x}\|_1.
\end{equation}
and then set any values of $x$ larger in absolute value than a parameter `threshold' to 0 and consider the resulting vector's support.
We do this $500$ times for each choice of $N \in \{10,20,\dotsc,10010\}$
inclusive. In addition, we also compute the condition number $\condfsul(b,U)$ and present the proportion of experiments for each dimension and distribution which had condition above $1,000$ as well as the median condition across all distributions of cases where Matlab fails to compute the correct support set and where Matlab computes the correct support set. The results are presented in Figure~1. 

 We see that
for large $N$, MATLAB is more accurate when data are drawn from an
exponential distribution instead of a normal distribution and
similarly the algorithm is more accurate when data are drawn from a normal
distribution instead of a uniform one. This also correlates with the size of the condition number: large median condition number correlates with low success rate.
\end{example}

\section{Main results} Our main results can be described with three main theorems: Theorem \ref{thm:failure_bounds} demonstrating  a relation between failure of finite precision algorithms and our condition number $\condfsul(b,U)$; Theorem \ref{thm:Prob1DCase}, which shows asymptotic estimates of $\condfsul(b,U)$, and thus helps explain Example \ref{Example:NDRandom}; and Theorem \ref{cor:simple} presenting conditions in classical statistics that allow us to obtain 'good/small' condition numbers $\condfsul(b,U)$ -- and thus effective and trustworthy algorithms for these inputs (we have a stronger, yet more involved version presented as Theorem \ref{thm:NormalCondWainwright} in \S \ref{section:SetupForNormDist}).

\subsection{Condition $\condfsul(b,U)$ and failure of algorithms on random inputs} The failure of MATLAB's {\tt lasso} in 
Example \ref{Example:NDRandom} and Figure \ref{fig:ExampleNDRandomResult} yields the question: Why does the algorithm fail on these basic random LASSO problems? The key issue is that $\condfsul(b,U)$ characterises failures of finite precision algorithms.

\begin{remark}[Finite precision algorithms]\label{rem:finite}
For a precise definition of a finite precision algorithm, see Definition \ref{definition:FPAlgorithm}. However, the concept can be explained simply: A finite precision algorithm with precision $2^{-k}$ (with $k \in \mathbb{N}$) can only read a dyadic approximation of the correct input with error bound $2^{-k}$.\end{remark}

\begin{remark}[Inputs that are represented exactly]\label{rem:Correct}
Given Remark \ref{rem:1}, there are certain dyadic inputs for which a finite precision algorithm will have an exact representation. Hence, we can consider algorithms that are \emph{correct on all inputs that can be represented exactly}. This concept is formally defined in Definition \ref{def;correct}.
\end{remark}

\begin{theorem}\label{thm:failure_bounds}
	Let $\Gamma_1$ and $\Gamma_2$ be finite precision algorithms with precision $2^{-k}$ with an open domain $\Omega \subset \real^m\times\real^{m\times N}$ for the LASSO function $\Xi$ defined in \eqref{eq:LassoComp}. Suppose that $\Gamma_1$ is correct on all inputs that can be represented exactly in $\Omega$ (see Remarks \ref{rem:finite} \& \ref{rem:Correct}) .  Suppose that $\condfsul(b,U) = \alpha$ with $0 < \frac{1}{\alpha} < 2^{-k-1}$ and that for some $r \in (\alpha^{-1},2^{-k-1})$ we have $\mathcal{B}_{\infty}(b,U,r) \subset \Omega$. Then, $\Gamma_1$ fails on a set $F$ such that \begin{itemize}
		\item $\mathcal{B}_{\infty}(b,U,r) \setminus F$ has Lebesgue measure 0,
		\item $\mathcal{B}_{\infty}(b,U,\alpha^{-1}) \in F$.
	\end{itemize}
	Moreover, $\Gamma_2$ either fails on the whole of  $\mathcal{B}_{\infty}(b,U,\alpha^{-1})$ or on another open set $\theta \subset \mathcal{B}_{\infty}(b,U,r) $. 
\end{theorem}

Theorem \ref{thm:failure_bounds} demonstrates how finite precision algorithms will fail when $\condfsul(b,U)$ is large relative to the precision of the algorithm. This means that if the probability that $\condfsul(b,U)$ is large is high, it is highly likely that the algorithm will fail. 

\subsection{Condition $\condfsul(b,U)$ as a random variable} 

Theorem~\ref{thm:Prob1DCase} provides a theoretical explanation
to the behaviours exhibited in Example \ref{Example:NDRandom} and the
corresponding Figure \ref{fig:ExampleNDRandomResult}.

\begin{theorem}\label{thm:Prob1DCase}
Let $y\in\real$ be fixed and $A \in \real^{1 \times N}$ be random
with i.i.d.~entries.  
\begin{itemize}
\item If the entries of $A$ follow
an exponential distribution with parameter~$1$ then
\[
  \lim_{N\to\infty}\probab\left(\frac{1}{t}< \condfsul(y,A)\right)=
  \begin{cases} 1-\rme^{-2t} &\mbox{for $t<|y|$}\\
		1 &\mbox{for $t\ge |y|$.}
  \end{cases}  
\]
\item If the entries of
$A$ follow the normal
distribution $\mcN(\mu,\sigma^2)$, then \newline
$2\sqrt{2\ln(N)}[\condfsul(y,A)]^{-1}$ converges in
distribution to an exponential random variable with
parameter~$1/\sigma$.
		
\item If the entries of $A$ follow a uniform
distribution $\mcU(0,1)$ on $(0,1)$, then
\newline
$2N[\condfsul(y,A)]^{-1}$
converges in distribution to an exponential random
variable with parameter $1$.
\end{itemize}
\end{theorem}

This theorem can help us to explain
Example~\ref{Example:NDRandom}. Indeed, a basic understanding of
Theorem~\ref{thm:Prob1DCase} is that, with high probability and for
large $N$, the condition number (as a function of $N$) stays roughly
constant when $A$ is exponentially distributed,
grows like $\sqrt{\ln(N)}/\sigma$ when $A$ is normally distributed,
and grows like $N$ when $A$ is uniformly distributed.

Note that -- according to Theorem \ref{thm:failure_bounds} -- larger values of the condition number are unfavourable for finite precision algorithms. Thus, it is unsurprising that in
Figure \ref{fig:ExampleNDRandomResult} we see that MATLAB commits more
errors for large $N$ when $A$ is uniform than when $A$ is normal and
similarly that MATLAB commits more errors for large $N$ when $A$ is
normal than when $A$ is exponentially distributed.

\subsection{{\bf $\condfsul(b,U)$ and trustworthy algorithms for classical statistics}}
 
Our next set of major results is focused on the following question:
\vspace{1mm}
\begin{displayquote}
\normalsize
{\it In view of Example \ref{Example:NDRandom}, Figure \ref{fig:ExampleNDRandomResult} and Theorem \ref{thm:failure_bounds}, under which conditions do there exist efficient and trustworthy finite precision algorithms -- for the LASSO feature selection problem with random data -- that produce correct outputs with high probability?}
\end{displayquote}
\vspace{1mm}
Our final main result (Theorem~\ref{thm:NormalCondWainwright}) gives a
bound on the probability of $\condfsul$ being large for
normally distributed inputs $(b,U)$ (as a function of
the input's size). Its statement involves a number of technical
conditions laid out in Section~\ref{section:SetupForNormDist} below. 
To give an idea of its significance, however, we next state a variant
of it in a simple, yet frequently considered, context. 
Consistently with the notation used in Part~I, we write
$\trip{v}_2:=\max\{1,\|v\|_2\}$.

\begin{theorem}\label{cor:simple}
Let $v\in\real^N$, $S=\supp(v)$, and $s=|S|$. Assume $\ln (N/2)\leq s\leq N/8$
and $m\le N/9$. 
Let $U\in\real^{m\times M}$ be random with i.i.d. entries with
standard Gaussian distribution and $b=Uv$. There exists a
universal constant $\bar{c}\ge 1$ with the following property. 
Assume that
\begin{enumerate}[label={\rm (\roman*)}]
\item
$m>(1+\epsilon)12s\ln(N-s)$ for some $\epsilon\in(0,1/2)$ with
$\epsilon>8\sqrt{s/m}$, 
\item
$\displaystyle \bar{c}\lambda<2m\min_{j\in S}|v_j|-\frac{1}{N^2}$,
\item
$\lambda\ge \frac{2}{N^2}$.
\end{enumerate}
Then, the algorithm in Theorem~I.1.2, which reads variable-precision
approximations of input $(b,U)$, returns
$S=\supp(v)$, with a cost bounded by $\Oh\big(N^3(\log_2 N\trip{v}_2)^2\big) $
and maximum number of digits bounded by 
$\Oh(\left\lceil \log_2\left(N\trip{v}_2\right)\right\rceil)$
with probability at least $1- \bfC_1 N^{-\bfC_2}$ for some positive universal constants $\bfC_1$ and $\bfC_2$.
\end{theorem}

\begin{remark}
The more general statement in Theorem~\ref{thm:NormalCondWainwright} applies under the presence of noise ---now $b=Uv+w$ with $w\in\real^m$ random with i.i.d~entries
drawn from $\mcN(0,\eta^2)$--- however, this requires slightly more involved assumptions and is done in the next section.
\end{remark}

\section{Precise and general statement of Theorem \ref{cor:simple}}
\label{section:SetupForNormDist}

Recall, for a matrix $M\in\real^{m\times N}$, its
$\infty$-norm is given by 
$\|M\|_{\infty}:= \max_{i \in \{1,2,\dotsc,m\}} \sum_{j=1}^{N} |M_{i,j}|$.

In this section we consider $m$ feature points $a_i\in\real^N$
independently drawn from a normal distribution in $\real^N$.
Moreover, we assume that the data $y$ is a random
corruption of a fixed linear predictor $v\in\real^N$.
More precisely, we consider the setup given by the following assumptions:
\begin{enumerate}[label = (S\roman*)]
\item
$A \in \real^{m \times N}$ is a random matrix such that each row is an
i.i.d.~random vector with distribution $\mcN(0,\Sigma)$ for some covariance
matrix $\Sigma$. 
\item
The vector $w \in \real^{m}$, chosen independently from $A$, has
i.i.d.~entries with $\mcN(0,\eta^2)$ distribution. 
\item
$y = Av + w$ where $v\in\real^N$ is non-random and has support
$S$ with $|S| = s$. 
\end{enumerate}

As it happens, we cannot expect a useful probabilistic bound on $\condfsul$
for arbitrary $m,N$, and $\Sigma$ even if there is no noise in the
measurements (i.e. $\eta^2 = 0$). 
In order to prove a useful bound on the condition number we will
therefore make some assumptions on the covariance matrix $\Sigma$,
number of measurements $m$, and vectors $v$. Since our intention is to
understand when the lasso can be successfully applied, it is sensible
to use existing conditions that guarantee a low chance of a recovery error.
Our goal will be to show that these conditions
can also be used to give an upper bound on the condition number and
thus a guarantee of a low chance of a numerical error. As a starting
point, we therefore use the conditions defined in an important paper
in understanding recovery errors for unconstrained
lasso,~\cite{Wainwright:09}.
More precisely, we make use of the following parameters
taken from~\cite{Wainwright:09}:
\begin{enumerate} 
\item
For sets $G = \{g_1,g_2,\dotsc,g_{|G|}\}\subseteq \{1,2,\dotsc,N\}$ and
$H = \{h_1,h_2,\dotsc,h_{|H|}\} \subseteq \{1,2,\dotsc,N\}$ we define
the matrix $\Sigma_{GH} \in \real^{|G| \times |H|}$ so that the $i,j$th entry $(\Sigma_{GH})_{i,j}$ is given by $\Sigma_{g_i,h_j}$
In other words, $\Sigma_{GH}$ is the restriction of $\Sigma$ to the rows given
by $G$ and the columns given by $H$.
\item
Define $C_{\min}$ and $C_{\max}$ to be, respectively, the minimal and
maximal eigenvalues of $\Sigma_{SS}$.
\item
Define the matrix $\Sigma_{S^{\mathsf{c}}|S} \in \real^{(N-s) \times (N-s)}$ by
$\Sigma_{S^{\mathsf{c}}|S}:= \Sigma_{S^{\mathsf{c}}S^{\mathsf{c}}}
- \Sigma_{S^{\mathsf{c}}S}(\Sigma_{SS})^{-1}\Sigma_{SS^{\mathsf{c}}}$.
\item We define
$\gamma :=   1-\|\Sigma_{S^{\mathsf{c}}S}(\Sigma_{SS})^{-1}\|_{\infty} $.
\item
For a square, symmetric, matrix $M$ we define
$\rho_l(M) := \min_{i \neq j} (M_{ii} + M_{jj} - 2M_{ij})/2$ and
$\rho_u(M) := \max\{M_{ii}\}$. When convenient, we will write for
shorthand $\rho_u = \rho_u(\Sigma_{S^{\mathsf{c}}|S})$ and
$\rho_l = \rho_l(\Sigma_{S^{\mathsf{c}}|S})$.
\item
Define $\theta_l = \theta_l(\Sigma) := \rho_l/(C_{\max} (2 - \gamma^2))$
and $\theta_u=\theta_u(\Sigma) := \rho_u/(C_{\min}\gamma^2)$,
\item Define \begin{equation}\label{eq:phiN}
	\phi_N:=\frac{\lambda^2}{8\eta^2\ln(N)C_{\min}\theta_u m}, 
\end{equation}
\end{enumerate}

\begin{remark}\label{remark:CminCmax}
Although at first glance these parameters
---$C_{\min},C_{\max},\theta_u,\theta_l$ and $\gamma$---
seem somewhat complicated, it can be revealing to consider their values 
when the rows of $A$ are drawn from a standard isotropic Gaussian, that is,
when $\Sigma=\Id_N$. In this case, for any $S$,  $\Sigma_{SS}=\Id_s$.
Therefore clearly,
$C_{\min}=C_{\max}=1$. Also, $\Sigma_{S^{\mathsf{c}}S}=0$ and 
$\Sigma_{S^{\mathsf{c}}S^{\mathsf{c}}} = \Id_{N-s}$ so that $\Sigma_{S^{\mathsf{c}}|S}=\Id_{N-s}$ 
and hence, 
$\gamma=\rho_l = \rho_u = 1$  and thus $\theta_l=\theta_u=1$. This means that 
$\phi_N = \lambda^2 (8 \eta^2 \log(N) m)^{-1}$.
\end{remark}
We next consider the following assumption:
\begin{description}
\item[(a0)]
$\gamma$ is strictly positive.
\end{description}

Our starting point is ~\cite[Thm.~3]{Wainwright:09}, which when written
in the notation of this paper becomes the following.

\begin{theorem}[Theorem 3~\cite{Wainwright:09}]\label{thm:WainwrightMainTheorem}
Assume that $\phi_N \geq 2$ and that
\begin{equation}\label{eq:mConditionWainwrightLasso}
		\frac{m}{2s\ln(N-s)} > (1+\epsilon) \theta_u
		\left(1+\frac{4m^2\eta^2 C_{\min}}{\lambda^2 s}\right)
\end{equation} for some $\epsilon \in (0,1/2)$ with
$\epsilon> \max\{8C_{\min} \sqrt{s/m},\sqrt{s/m}\}$ as well as assumption (a0).
 
Then there exist constants $c_1,c_2$ and $c_3$ independent of all
parameters such that the following is true. If
\[
  g(\lambda) := \frac{c_3 \lambda\|\Sigma^{-1/2}\|_{\infty}^2} {2m}
  + 20 \sqrt{\frac{\eta^2\ln(s)}{C_{\min}m}} < \min_{j\in S}|v_j|
\] 
then, with probability greater than $1-c_1 \rme^{-c_2 \min\{s,\ln(N-s)\}}$, 
the LASSO problem $\mUL(y,A)$ defined in \eqref{eq:TheSolSet} has a unique solution $x$. Moreover
$x$ satisfies the following:
\begin{enumerate}[label={\rm (\arabic*)}]
\item $\supp(x) = S$ 
\item $\sg(x_S) = \sg(v_S)\in\{-1,1\}^S$
\item $\|x_S - v_S\|_{\infty} \leq g(\lambda)$.\eproof
\end{enumerate}
\end{theorem}

Theorem \ref{thm:WainwrightMainTheorem} is a very important result in
the literature concerning the lasso. Firstly, the result applies both
when $\Sigma = \Id_N$ and $\eta=0$ (giving the traditional
bound from the compressed sensing literature that $m \geq Cs\log(N-s)$
for some constant $C$ as a sufficient condition for recovering the set
$S$ from gaussian measurements) and when $\Sigma \neq \Id_N$ to
explain the more realistic scenario of trying to distinguish between
correlated normally distributed features.

The second and perhaps more crucial reason for the importance of
Theorem \ref{thm:WainwrightMainTheorem} is
optimality. Indeed, \cite[Thm.~4]{Wainwright:09} also contains
a corresponding lower bound: if instead of
assuming \eqref{eq:mConditionWainwrightLasso}, we assume
that \begin{equation*} \frac{m}{2s\ln(N-s)} <
(1-\epsilon) \theta_l \left(1+\frac{4m^2\eta^2
C_{\max}}{\lambda^2 s}\right)
\end{equation*} 
then (with probability approaching $1$ as $N$ increases) no solution
$x$ of the lasso problem has $\supp(x) = S$ and $\sg(x_S)
= \sg(v_S)$. Thus the hypotheses of
Theorem~\ref{thm:WainwrightMainTheorem} can be
seen as necessary for the lasso to avoid recovery errors and are
therefore a basic requirement for working with the lasso with normally
distributed data. 

It is therefore worth considering if the conditions in
Theorem \ref{thm:WainwrightMainTheorem} and assumption (a0) are
sufficient to also avoid numerical errors. The major result of this
section will be to show that this is indeed the case, under slightly
stronger assumptions. More precisely, we assume the following:
\begin{description}
\item[(ai)]
The number of measurements $m$ satisfies
\begin{equation}\label{eq:normalMRequiredMeasurements}
	m\left(\frac{1}{1+\epsilon} - \frac{6}{\phi_N}\right)
        > 12s\ln(N-s)\theta_u
\end{equation}
for some $\epsilon \in (0,1/2)$ with 
$\epsilon > \max\{8C_{\min} \sqrt{s/m},\sqrt{s/m}\}$.
\item[(aii)]
$\displaystyle g(\lambda) < \min_{s\in S}|v_s|$.
\end{description}

At first glance, it may seem that assumption (ai) is only loosely
related to \eqref{eq:mConditionWainwrightLasso}. However, because
of~\eqref{eq:phiN}, we can write \eqref{eq:mConditionWainwrightLasso}
as
\begin{equation*}
	\frac{m}{2s\ln(N-s)} >(1+\epsilon) \theta_u
	\left(1+\frac{m}{2s\ln(N)\theta_u\phi_N}\right)
\end{equation*}
Replacing the constant $1/2$ in the left-hand side by the smaller
value $1/12$ and replacing $1/\ln(N)$ by the larger value $1/\ln(N-s)$
in the right-hand side of this bound yields the simpler
inequality \eqref{eq:normalMRequiredMeasurements}.
Thus \eqref{eq:normalMRequiredMeasurements}
can be seen as a slightly stronger condition
than \eqref{eq:mConditionWainwrightLasso} (but note
that \eqref{eq:normalMRequiredMeasurements} is still optimal up to the
change of constants). Moreover, for the left hand side
of \eqref{eq:normalMRequiredMeasurements} to be positive we will
require $\phi_N \geq 6(1+\epsilon)$ and hence assumption (ai)
supersedes the requirement that $\phi_N \geq 2$.

To state our result under assumptions (ai--aii) we will use one
additional parameter. Recall from (I.7.1)
the function $q$ in the variables $\nu,\xi>0$,
\begin{equation}\label{eq:q}
	q(\nu,\xi) := 96\nu^5+12\nu^3(1+\lambda \sqrt{N})\sqrt{\xi}
	+ \xi \left(\frac{2\nu^3}{\lambda} + 3\nu\right).
\end{equation}
We are now in a position to state the major result of this section.

\begin{theorem}\label{thm:NormalCondWainwright}
In the setup described by {\rm (Si--iii)} and under the
assumptions {\rm (ai--aii)},
there exists constants $\bfc_1,\bfc_2$ such that if $p = \bfc_1 \rme^{-\bfc_2\min\{\ln(N-s),s\}}$ then
\[
\probab(\condfsul(y,A) \geq \widehat{K} ) \leq p
\text{ where }
\widehat{K}:=(mN)^{\frac{1}{2}}\max\left\{\frac{q(\hat{\alpha},
	\hat{\sigma})}{\hat{\sigma}^2},
	\frac{6\hat{\alpha}}{\sqrt{\hat{\sigma}}},1 \right\}
\]
with
\begin{align*}
	&\hat{\sigma} =
        \min\Big\{C^2_{\min}/(4(\sqrt{s}+\sqrt{m})^4),\lambda/2,
	{\displaystyle\min_{j\in S}}|v_j| - g(\lambda)\Big\},\\
	&\hat{\alpha} = \max\big\{1,\sqrt{2m}(\eta + C_{\max}\|v\|_2),
	\|\Sigma\|_2^{1/2}(3\sqrt{m}+6\sqrt{N})\big\}.
\end{align*}

In particular, the algorithm in Theorem~I.1.2, which reads variable-precision
approximations of input $(y,A)$, returns
$S=\supp(v)$, with a cost bounded by $\Oh\big(N^3(\log_2 \hat K\trip{v}_2)^2\big)$
and maximum number of digits bounded by 
$
\Oh(\left\lceil \log_2\left(\hat K\trip{v}_2\right)\right\rceil).
$
with probability at least $p$.
\end{theorem}

As per Remark \ref{remark:CminCmax}, when the rows of $A$ are drawn from a
standard isotropic Gaussian assumption (a0) is automatically satisfied
whereas (ai--aii) reduce to the following
\begin{description}
\item[(ai')]
The number of measurements $m$ satisfies
\begin{equation}
	m\left(\frac{1}{1+\epsilon} - \frac{6}{\phi_N}\right) > 12s\ln(N-s)
\end{equation}
for some $\epsilon \in (0,1/2)$ with 
$\epsilon > 8\sqrt{s/m}$.
\item[(aii')]
$\displaystyle g(\lambda) < \min_{s\in S}|v_s|$ where 
$g(\lambda) = c_3\lambda/(2m) + 20 \sqrt{\eta^2 \ln(s)/m}$ and $c_3$
is as in Theorem~\ref{thm:WainwrightMainTheorem}.
\end{description}
and hence we obtain the following simpler form of
Theorem~\ref{thm:NormalCondWainwright} which we state in full.

\begin{corollary}\label{cor:isotropic}
Assume $\Sigma=\Id_N$. In the setup above and under assumptions
{\rm (ai')} and {\rm (aii')},
there exist absolute constants $\bfc_1$ and $\bfc_2$
such that the following holds true. Let 
\begin{align*}
&\hat{\sigma} = \min\left\{\frac{1}{4(\sqrt s+ \sqrt m)^4},\frac{\lambda}{2},
		{\displaystyle\min_{j\in S}}|v_j| - g(\lambda)\right\}\\
		\mbox{and\quad}
&\hat{\alpha} = \max\Big\{\sqrt{2m}(\eta + \|v\|_2),
3\sqrt{m}+6\sqrt{N}\Big\}.
\end{align*}
Then,
$
 \probab(\condfsul(y,A) \geq \widehat{K} ) \leq \bfc_1 \rme^{-\bfc_2\min\{\ln(N-s),s\}}
$  
where $
\widehat{K}:=(mN)^{\frac{1}{2}}\max\left\{\frac{q(\hat{\alpha},\hat{\sigma})}
{\hat{\sigma}^2},\frac{6\hat{\alpha}}{\sqrt{\hat{\sigma}}}\right\}.$
In particular, the algorithm in Theorem~I.1.2, which reads variable-precision
approximations of input $(y,A)$, returns
$S=\supp(v)$, with a cost bounded by $\Oh\big(N^3(\log_2 \hat K\trip{v}_2)^2\big)$
and maximum number of digits bounded by 
$
\Oh(\left\lceil \log_2\left(\hat K\trip{v}_2\right)\right\rceil).
$
with probability at least $p$.
\end{corollary}

\section{Connection to previous work} Below follows an account of
the connection to different areas and works that are crucial for the paper. 

\begin{itemize}[leftmargin=5pt]
	\item[] \emph{Condition in optimisation:} The concept of condition numbers has proven a crucial to
	computational mathematics and numerical analysis for securing trustworthy algorithms that are accurate and stable~\cite{Higham96, Cucker_Smale97}. 
	J. Renegar's contributions in optimization and condition are particularly noteworthy, with~\cite{Renegar1, Renegar2, renegar1988polynomial,
		renegar2001mathematical, Renegar96} important for understanding
	stability, accuracy, and efficiency of optimization algorithms. This is extensively discussed in~\cite{Condition}. Furthermore, the following have important links to condition in optimization:
	J. Pe{\~n}a~\cite{Pena2002, Pena2001} as well as D. Amelunxen,
	M. Lotz, J. Walvin~\cite{Lotz2020}, and D. Amelunxen, M. Lotz,
	M. McCoy, J. Tropp~\cite{Lotz2014} see also~\cite{ChCP07,
		felipe_cond_01, ChC02}.
	\vspace{1mm}
	
	\item[] \emph{GHA and robust optimisation:}
	GHA~\cite{opt_big, CSBook, gazdag2022generalised, comp_stable_NN22,
		paradox22} plays an instrumental role in establishing some of the computational barriers that this paper introduces. The topic is related (although mathematically very different)
	to hardness of approximation in computer
	science~\cite{Arora2007}. Importantly, GHA in optimisation can be viewed
	as a part of the broader program on robust optimisation (A. Ben-Tal, L. El
	Ghaoui \& Nemirovski~\cite{Nemirovski_robust, NemirovskiLRob,
		Nemirovski_robust2}) for computing minimisers. It is also a
	part of broader efforts to establishthe mathematics behind the Solvability Complexity
	Index (SCI) hierarchy, see for example the work by J. Ben-Artzi,
	M. Colbrook, M. Marletta~\cite{Hansen_JAMS, Ben_Artzi2022, SCI,
		Matt1}.
	
	\vspace{1mm}
	
	\item[] \emph{Trustworthy algorithms and computer assisted proofs:}
	The importance of finding trustworthy algorithms for optimisation goes beyond scientific computing: it has important implications in computer assisted proofs. T. Hales' proof of Kepler's
	conjecture~\cite{Hales1, Hales2} is a prime example.  Hales' computer assisted proof of this famous conjecture relied on solving around 50,000 linear
	programs with irrational inputs, thus highlighting the importance of understanding 
	computation with inexact inputs across all of mathematics. For other examples, see~\cite{AIM}. Of particular interest to this work are Problem 2 (T. Hou) and Problem 5 (J. Lagarias) which discusses results on
	developing algorithms that are 100\% trustworthy and thus appropriate for
	computer assisted proofs.
	
	\vspace{1mm}
	
	\item[] \emph{Algorithms for computing minimisers of LASSO:}   Numerous algorithms are available for solving the LASSO problem, as detailed in the review articles by Nesterov \& Nemirovski \cite{Nesterov_Nemirovski_Acta} and Chambolle \& Pock \cite{chambolle_pock_2016}, which include additional references for a comprehensive understanding. See also the work by Beck \& Teboulle \cite{Fista} and Wright, Nowak \&  Figueiredo \cite{Mario_Lasso}, and the references therein, as well as \cite{Boyd, wright2022optimization, CSBook}. However, while these algorithms are capable of approximating the objective function, they cannot -- in general -- determine the support sets of minimisers of LASSO. 
\end{itemize}

\section{Definitions and results from \cite{BCH1}}
In this section we recall some definitions and results that are taken from~\cite{BCH1}. These are provided without proofs, as the proofs are contained in~\cite{BCH1}, and are included for the sake of ensuring that this paper can be read as self contained material.

In addition to $\triple{y,A}_{\max}$ we will also use the $\ell^p$-norm
$\|y\|_{p}$ of $y$, the operator norms 
$\|A \|_{qr} = \sup_{\|x\|_q = 1} \|Ax\|_r$  (writing $\|A\|_q$ when $q=r$),
and the truncated norms 
\begin{equation*}
	\nTrOneStar{y,A}= \max\left\{\sum_{i=1}^{m}\sum_{j=1}^{N} |A_{ij}|,\,
	\sum_{i=1}^{m}  |y_i|,\,1\right\}
\end{equation*}
and $\nTrTwo{y,A}:=\max \left\{\|A\|_2,\|y\|_2,1\right\}$.

\begin{definition}\label{def:stabSupp}
	The \emph{stability support} of a pair $(y,A)$ is defined as 
	\begin{align*}
		\stabsupp(y,A):=\inf \Big\lbrace& \delta \geq 0 \, \big\vert \, 
		\exists\,  \tilde{y} \in \real^{m},\tilde{A} \in \real^{m \times N},
		x\in \mUL(y,A), \mbox{ and } \tilde{x} \in \mUL(\tilde{y},\tilde{A})\\ 
		&\!\!\mbox{ such that } \|\tilde{y} - y\|_{\infty}, \nmax{A - \tilde{A}}
		\leq \delta \mbox{ and }  \supp(x) \neq \supp(\tilde{x})\Big\rbrace.
	\end{align*}
\end{definition} 
The stability support is therefore the {\em distance to support change}.  
If $\stabsupp(y,A)>0$ then there exists $S\in{\mathbb B}^N$ such that 
$\Xi(y,A)=\{S\}$. Furthermore, for all pairs $(y',A')$ in a ball
(w.r.t.~the max distance) of radius $\stabsupp(y,A)$ around $(y,A)$
we have $\Xi(y',A')=\{S\}$. If, instead, $\stabsupp(y,A)=0$ then there are
arbitrarily small perturbations of $(y,A)$ which yield LASSO solutions with
different support. We use stability support to define the condition:

\begin{definition}\label{def:cond}
	For an input $(y,A)$ to UL feature selection we define the {\em condition
		number} $\condfsul(y,A)$ to be
	\begin{equation}\label{eq:CondDefinition}
		\condfsul(y,A) =
		\begin{cases} 
			(\stabsupp(y,A))^{-1} & \text{if } \stabsupp(y,A) \neq 0\\
			\infty & \text{otherwise.}
		\end{cases}
	\end{equation}
	The set $\Sigmul:=\{(y,A)\mid \stabsupp(y,A)=0\}$ is the
	{\em set of ill-posed inputs}. 
\end{definition}

To make this easier to work with, we define the following:
\begin{definition}\label{def:sigma1sigma2sigma3}
	For a pair $(y,A)$, we write (with the convention that if $M$ is
	non-invertible, $\nTwoTwo{M^{-1}}:=\infty$ and so $\nTwoTwo{M^{-1}}^{-1} = 0$),
	\begin{align*}
		\sigma_1(y,A)&:= \inf \lbrace t \, \vert \, \exists x \in \mUL(y,A)
		\text{ with } 
		\|A_{S^{\mathsf{c}}}^*(Ax-y)\|_{\infty}= \lambda/2 - t, S = \supp(x) \rbrace, \\
		\sigma_2(y,A)&:=   \inf \lbrace \nTwoTwo{(A^*_SA_S)^{-1}}^{-1} \, \vert \,
		\exists x \in \mUL(y,A) \text{ with } S = \supp(x) \rbrace, \\
		\sigma_3(y,A)&:= \inf \lbrace t \, \vert \, \exists
		i \in \lbrace 1,2,\dotsc,N \rbrace  
		\text{ and } x \in \mUL(y,A) \text{ such that }  0 < |x_i| \leq t \rbrace.
	\end{align*}
	where, for the empty-set $\varnothing$, we interpret
	$\|A^*_{\varnothing}(Ax-y)\|_{\infty} = 0$, we treat $A^*_\varnothing A_\varnothing$
	as invertible with $\nTwoTwo{(A^*_\varnothing A_\varnothing)^{-1}}^{-1} = \infty$,
	and we set $\inf \varnothing = \infty$. 
\end{definition}
We combine each of $\sigma_1, \sigma_2$ and $\sigma_3$ into a single
quantity as follows,
\[
\sigma(y,A) := \min\left\{\sigma_1(y,A),\sigma_2(y,A)^2,\sigma_3(y,A)\right\}
\] 
The next proposition provides a lower bound for $\stabsupp$ which makes use of the polynomial \eqref{eq:q}:
\begin{proposition}\label{proposition:rhofssigmalb}
	Set $\alpha = \nTrTwo{y,A}$ and 
	$\sigma = \sigma(y,A)$. Then
	$
	\stabsupp(y,A) \geq (mN)^{-\frac{1}{2}}\min\left\{\frac{\sigma^2}
	{q(\alpha,\sigma)},
	\frac{\sqrt{\sigma}}{6\alpha},\alpha \right\}.
	$
\end{proposition}

\section{Proof of Theorem~\ref{thm:Prob1DCase}}

To prove Theorem~\ref{thm:Prob1DCase}, we first relate the condition
number to some quantities that will be easier to deal with.
Specifically, we define the event $Z$ and the quantity $\delta$ as follows.

\begin{definition}
For pairs $(y,A)\in\real^{1+N}$ and $\epsilon > 0$, we define
$Z(\epsilon) =Z(y,A,\epsilon)$ 
to be the following event: if
$\|(y,A)-(\tilde{y},\tilde{A})\|_{\max}\le\epsilon$, then 
$0\not\in\mUL(\tilde{y},\tilde{A})$.

\end{definition}

\begin{definition}
For a vector $A\in\real^N$ we let $\delta\geq 0$
denote the difference between the largest entry of $|A|$ and the
second largest entry of $|A|$, where $|A| \in \real^N$ is
the vector with entries $|A_i|$ for $i=1,\ldots,N$.
\end{definition}

The proof of Theorem~\ref{thm:Prob1DCase} easily follows from
the following four lemmas, the second of
which relates the condition number to $Z$ and $\delta$.

\begin{lemma}\label{lemma:1DZeroSupportRequirement}
Let $A \in \real^{N}$ and $i \in \{1,2,\dotsc,N\}$ be such that
$|A_{i}| > |A_{j}|$ for all $j \in \{1,2,\dotsc,N\}$ with $j\ne i$.
Let $y\in\real$ and $x \in \mUL(y,A)$. Then $\supp(x) = \{i\}$ if
$|A_{i}y| > \lambda/2$ and $\supp(x) = \varnothing$
(i.e. $x= 0$) if $|A_{i}y| \leq \lambda/2$.
\end{lemma}

\begin{lemma}\label{lemma:ZAndDeltaConditionRelation}
Let $A\in\real^N$ be randomly drawn from an absolutely continuous
distribution with respect to the Lebesgue measure in $\real^N$.
Let $y\in\real$ and $\epsilon > 0$. Then, 
$\probab[Z(\epsilon) \cap (\xul < \epsilon)] =
\probab[Z(\epsilon) \cap (\delta < 2\epsilon)]$.
\end{lemma}

Thus to understand $\xul$ it suffices to analyse $Z$ and
$\delta$. This is what we do in the next two lemmas under the
assumptions that $A$ is exponentially, gaussian, or uniformly
distributed.
		
\begin{lemma}\label{lemma:1DZ}
Let $F$ be the cumulative distribution of a non-negative random variable $X$
which is absolutely continuous with respect to the Lebesgue measure
in $\real$ and $c>0$ be such that $F(c)<1$. 
Suppose that $A^{(N)} \in \real^N$ is a random vector with i.i.d.~entries
such that $\big|A^{(N)}_i\big|$ is distributed as $X$. 
Let $y\in\real$ and $(b_N)_{N=1}^{\infty}$ be a sequence of
non-negative reals satisfying that $b_N < |y|$ and
$\lambda/(2|y|-2b_N) + b_N \leq c$ for all $N$ sufficiently large. Then 
$\lim_{N \to \infty} \probab(Z(y,A^{(N)},b_N)) = 1$.

In particular, for all $y\in\real$, and $b>0$
and all random $A\in\real^N$,
\begin{enumerate}[label={\rm (\arabic*)}]
\item
$\lim_{N \to \infty}\probab(Z(y,A,b)) = 1$ for $b < |y|$ and
$\lim_{N \to \infty}\probab(Z(y,A,b)) =0$ for $b \geq |y|$, if
each entry of $A$ is exponentially distributed with
parameter~$1$.\label{item:ZForExpA}
\item
$\lim_{N \to \infty} \probab(Z(y,A,b\sigma/\sqrt{2\ln(N)}))= 1$, if
$y\neq 0$ and each entry of $A$ is Gaussian with mean $1$, variance $\sigma^2$.
\label{item:ZForGaussian}
\item
$\lim_{N \to \infty} \probab(Z(y,A,b/N)) =1$, if $|y| > \lambda/2$ and each
entry of $A$ is uniformly distributed on $[0,1]$.\label{item:ZForUniform}
\end{enumerate}
\end{lemma}

\begin{lemma}\label{lemma:1DDelta}
Suppose that $A \in \real^N$ is a random vector with i.i.d.~entries.
Then for each $t >0$ 
\begin{enumerate}[label={\rm (\arabic*)}]
\item
$\lim_{N \to \infty} \probab(\delta > t) = \rme^{-t}$
if each entry of $A$ is exponentially distributed.
\item
$\lim_{N \to \infty}  \probab(\delta > t\sigma/\sqrt{2\ln(N)}) = \rme^{-t}$ if each
entry of $A$ is Gaussian.
\item
$\lim_{N \to \infty} \probab(\delta > t/N) = \rme^{-t}$ if each entry of $A$
is uniformly distributed on $[0,1]$.
\end{enumerate}
\end{lemma}

Assuming (for now) that these Lemmas hold, we proceed to prove
Theorem~\ref{thm:Prob1DCase}.

\begin{proof}[Proof of Theorem~\ref{thm:Prob1DCase}]
For any sequence of positive real valued functions
$(f_n)_{n=1}^{\infty}$, we split $\probab(\xul < f_N(t))
= \probab[(\xul < f_N(t)) \cap Z(f_N(t))] + \probab[(\xul <
f_N(t)) \cap Z(f_N(t))^{\mathsf{c}}]$. Using 
Lemma~\ref{lemma:ZAndDeltaConditionRelation}
with $\epsilon=f_N(t)$ on the first term
this splitting becomes 
\begin{equation}\label{eq:rhoSplittingDeltaZ}
  \probab(\xul < f_N(t)) = \probab[(\delta < 2f_N(t)) \cap Z(f_N(t))]
  + \probab[(\xul < f_N(t)) \cap Z(f_N(t))^{\mathsf{c}}]
\end{equation}
Now take $f_N(t)=t$ when the entries of $A$ are exponentially
distributed and $t<|y|$, $f_N(t)=\frac{t\sigma}{\sqrt{2\ln N}}$ if they are Gaussian,
and $f_N(t)=\frac{t}{N}$ if they are uniformly distributed on $[0,1]$.
In the three cases Lemma~\ref{lemma:1DZ} shows that
$\probab(Z(f_N(t)))\to 1$ when $N\to\infty$. Consequently, in all three cases 
we have
\[
  \lim_{N \to \infty} \probab(\xul < f_N(t))
  = \lim_{N \to \infty} \probab[(\delta < 2f_N(t)) \cap Z(f_N(t))]
  = \lim_{N \to \infty} \probab[\delta < 2f_N(t)].
\]
We can then apply Lemma~\ref{lemma:1DDelta} to deduce that
\begin{itemize}
\item
$\lim_{N \to \infty} \probab(\xul < t) = 1 - \rme^{-2t}$ for $t < |y|$
if each entry of $A$ is exponentially distributed with parameter~$1$.
\item
$\lim_{N \to \infty}  \probab(\xul < t/\sqrt{2\ln(N)}) = 1-\rme^{-2t/\sigma}$ if
each entry of $A$ is a standard Gaussian. 
\item
$\lim_{N \to \infty} \probab(\xul < t/N) = 1 - \rme^{-2t}$ if $|y| > \lambda/2$
and each entry of $A$ is  uniformly distributed on $[0,1]$.
\end{itemize}

It only remains to deal with $\probab(\xul < t)$ when
$t \geq |y|$ and each entry of $A$ is exponentially distributed with
parameter $1$. We start by examining $\probab[(\xul \geq t) \cap
Z(t)^{\mathsf{c}}]$ in this setting.
			
Assume that $Z(t)^{\mathsf{c}}$ occurs. Then $0 \in \mUL(\tilde{y},\tilde{A})$
for some perturbation $(\tilde{y},\tilde{A})$
of $(y,A)$ satisfying that
$d[(\tilde{y},\tilde{A}),(y,A)]_\infty < t$. If, in addition, 
$\xul \geq t$ then $0 \in \mUL(y,A)$ as well. 
Now, Lemma~\ref{lemma:1DZeroSupportRequirement} and
the fact that $|y| > 0$  show that this occurs
if and only $\max_{i=1,2,\ldots,N}
|A_i| \leq \lambda/(2|y|)$. Since the
entries of $A$ are independent, we conclude that 
$$
\probab (0 \in \mUL(y,A))
= \Pi_{i=1}^{N}\probab(|A_i| \leq \lambda/(2|y|))
= [1-e{-\lambda/(2y)}]^{N}
$$
a quantity tending to 0 as $N \to\infty$. Thus
$\lim_{N \to \infty} \probab[(\xul \geq t) \cap Z(t)^{\mathsf{c}}]
\leq \lim_{N \to \infty} \probab(0 \in \mUL(y,A))= 0$.
			
This last limit implies that
$\lim_{N \to \infty} \probab[(\xul <t) \,\cap \,Z(t)^{\mathsf{c}}]
= \lim_{N \to \infty} \probab(Z(t)^{\mathsf{c}})$.
As $t\geq|y|$, Lemma~\ref{lemma:1DZ} part (1) shows that
$\lim_{N \to \infty} \probab(Z(t)^{\mathsf{c}})=1$.
Hence $\lim_{N \to \infty} \probab(\xul <t) \geq \lim_{N \to \infty}
\probab[(\xul < t) \cap Z(t)^{\mathsf{c}}]
= \lim_{N \to \infty} \probab(Z(t)^{\mathsf{c}})=1$,
completing the proof.
\end{proof}

We now prove Lemmas~\ref{lemma:1DZeroSupportRequirement} to~\ref{lemma:1DDelta}.

\begin{proof}[Proof of Lemma~\ref{lemma:1DZeroSupportRequirement}]
The proof follows from the KKT conditions (UL5)
in Section~I.4.
Suppose firstly that $|A_{i}y| > \lambda/2$. Then $0$ is not
a solution to the lasso problem. Indeed, if it were the KKT conditions
would imply that $|A_i y|= |(A^*(A0-y))_i| =\lambda/2$, contradicting
our assumption. Moreover for all $j \in \supp(x)$ we have that 
$j$ is in the equicorrelation set and thus
$|A_j||(Ax-y)| = |A^*_j(Ax-y)| = \lambda/2 $. This means that the value
of $|A_j|$ is the same for all $j\in\supp(x)$. Our hypothesis then
imply that $\supp(x)=\{i\}$. 

Suppose now that, instead, $|A_i y| \leq \lambda/2$.  Then
$\|A^*(A0-y)\|_{\infty} = |A^*_i(A0-y)| \leq \lambda/2$ and thus $0$
is a solution as it satisfies the KKT conditions. 
We claim that $0$ is in fact the only solution. Indeed, 
by (UL4), all lasso solutions have the same $\ell^1$
norm and hence all solutions have norm $0$. Thus the
only possible solution is~$0$.
\end{proof}

\begin{proof}[Proof of Lemma~\ref{lemma:ZAndDeltaConditionRelation}]
Write, for simplicity, $\xi:=\xul(y,A)$. It suffices to show that
\begin{equation}\label{eq:TwoInequalitiesZDeltaCondRelation}
\probab[(\xi < \epsilon) \cap Z(\epsilon)] \leq
\probab[(\delta < 2\epsilon) \cap Z(\epsilon)] \mbox{ and }
\probab[(\xi \geq \epsilon) \cap Z(\epsilon)] \leq
 \probab[(\delta \geq 2\epsilon) \cap Z(\epsilon)].
\end{equation}
As the distribution for the entries of $A$ is absolutely continuous with
respect to the Lebesgue measure on $\real^{N}$, we can assume that each
entry of $A$ is unique (this is true with probability $1$). Assume
that $Z(\epsilon)$ occurs and let $x \in \mUL(y,A)$. Because $Z(\epsilon)$
holds,
$x\ne 0$. This implies the existence of $i\le N$ with $\supp(x)= \{i\}$ and
such that 
$|A_i| >|A_k|$ for all $k\ne i$ by Lemma~\ref{lemma:1DZeroSupportRequirement}.
			
Let us now prove the first of the inequalities
in \eqref{eq:TwoInequalitiesZDeltaCondRelation}. For this, assume that,
in addition to $Z(\epsilon)$, we have $\xi < \epsilon$. Then there exists
$(\tilde{y},\tilde{A})$ with
$d_{\max}[(\tilde{y},\tilde{A}),(y,A)] < \epsilon$ such that there is an
$\tilde{x} \in\mUL(\tilde{y},\tilde{A})$ with $\supp(\tilde{x})\neq\supp(x)$.
Under the assumption that $Z(\epsilon)$ occurs,
$\supp(\tilde{x}) \neq \varnothing$ and so
there must exist a $j$ such that $|\tilde{A}_j| \geq
|\tilde{A}_i|$ (otherwise Lemma~\ref{lemma:1DZeroSupportRequirement} applied
to $(\tilde{y},\tilde{A})$ implies that
$\supp(\tilde{x}) = \{i\}$, contradicting the fact that
$\supp(\tilde{x}) \neq \supp(x)$). Since $\tilde{A}$ is an
$\epsilon$-perturbation of $A$, the condition $|\tilde{A}_j|\geq
|\tilde{A}_i|$ implies $|A_j| \geq |A_i| - 2\epsilon$. In
particular, $\delta \leq 2\epsilon$, and we conclude that
$\probab(\xi < \epsilon \cap
Z(\epsilon)) \leq \probab(\delta < 2\epsilon \cap Z(\epsilon))$.
			
We proceed with the second inequality.
For this, assume that, in addition to $Z(\epsilon)$, we have 
$\xi \geq \epsilon$. Then for any $j \neq i$, the perturbation
$\tilde{A}$ defined by
$\tilde{A}_i = A_i-\epsilon \sg(A_i)$,
$\tilde{A}_j = A_j+ \epsilon \sg(A_j)$ and
$\tilde{A}_k = A_k$ whenever both $k \neq i$ and $k \neq j$ must be such that
if $\tilde{x} \in \mUL(y,\tilde{A})$ then
$\supp(\tilde{x}) = \{i\}$ by the definition of $\xi=\stabsupp(y,A)$.
But for this to occur we must have
$|\tilde{A}_i|\geq |\tilde{A}_j|$, otherwise
Lemma~\ref{lemma:1DZeroSupportRequirement} applies to yield either
$\tilde{x}=0$ or $\supp(\tilde{x}) = \{j\}$. The condition
$|\tilde{A}_i|\geq|\tilde{A}_j|$ reduces to
$|A_i| -\epsilon \geq |A_j|+ \epsilon$. Since this occurs for any
$j \neq i$, we must have $\delta \geq 2\epsilon$. We conclude that if
both $Z(\epsilon)$ and $\xi \geq \epsilon$ then we must have
$\delta \geq 2\epsilon$ and thus $\probab[(\xi \geq \epsilon) \cap
Z(\epsilon)] \leq \probab[(\delta \geq 2\epsilon) \cap Z(\epsilon)]$.
\end{proof}

\begin{proof}[Proof of Lemma~\ref{lemma:1DZ}]
Assume $Z\big(y,A^{(N)},b_N\big)$ does not hold. Then there exists
$(\tilde{y},\tilde{A})$ in the ball of radius $b_N$ (w.r.t.~$d_{\max}$) about
$\big(y,A^{(N)}\big)$ such that $0\in\mUL\big(\tilde{y},\tilde{A}\big)$.
This implies that   
\begin{equation}\label{eq:ball1}
    \big\|A^{(N)}\big\|_{\infty} < \|\tilde{A}\|_{\infty} + b_N
     \qquad\mbox{and}\qquad
     |\tilde{y}| > |y| - b_N
\end{equation}
and that (by Lemma~\ref{lemma:1DZeroSupportRequirement})  
$\|\tilde{A}\|_{\infty} |\tilde{y}| \le \lambda/2$. Together
with~\eqref{eq:ball1} this inequality implies that
$$
  \big\|A^{(N)}\big\|_\infty < \|\tilde{A}\|_{\infty} + b_N
  \le \frac{\lambda}{2|\tilde{y}|} +b_N < 
  \frac{\lambda}{2(|y|-b_N)} + b_N.
$$
Consequently, for sufficiently large $N$,
\begin{eqnarray*}
\probab\big(Z\big(y,A^{(N)},b_N\big)\big) &\ge&
\probab\bigg(\big\|A^{(N)}\big\|_\infty >\frac{\lambda}{2|y|+2b_N} - b_N\bigg)\\
&=& 1- F\bigg(\frac{\lambda}{2|y|+2b_N} - b_N\bigg)^N \ge 1-F(c)^N 
\end{eqnarray*}
the equality because the entries of $A^{(N)}$ are i.i.d.~and our hypothsesis
on their distribution, and the last inequality by our hypothesis.
We conclude (since we assume that $F(c) < 1$) that
\[
\lim_{N \to \infty} \probab\big(Z\big(y,A^{(N)},b_N\big)\big) = 1.
\]
			
We use this result to easily prove each
of~\ref{item:ZForExpA},~\ref{item:ZForGaussian}
and~\ref{item:ZForUniform}. Starting with~\ref{item:ZForExpA}, suppose
first that $b < |y|$. The result follows from setting
$b_N=b$ for all $N\ge 0$ and  
$c=\lambda/(2|y|-2b)+b+1$. In the case $b \geq |y|$ consider the 
perturbation $\tilde{y} = 0$, $\tilde{A} = A$. Then
$|y-\tilde{y}|=|y| \leq b$ and, obviously, $\|\tilde{A}-A\|_\infty=0\leq b$.
But $0$ is the unique solution to lasso for the pair $(\tilde{y},\tilde{A})$, 
so $Z(y,A,b)$ does not hold and, consequently, $\probab(Z(y,A,b)) = 0$.
			
Similarly, in the Gaussian case~\ref{item:ZForGaussian} we set 
$c:= \frac{\lambda}{|y|}$ and $b_N:=\frac{b\sigma}{\sqrt{2\ln N}}$. 
Clearly, $c<\infty$ so that $F(c)<1$. In addition, for $N$ sufficiently large,
$b_N < |y|$ and $\lambda/(2|y|-2b_N) + b_N < c$, both because $b_N\to0$ when
$N\to\infty$. The desired convergence follows. 

Finally, in the uniform case~\ref{item:ZForUniform} we set
$c = 1/2 + \lambda/(4|y|)$ (i.e. $c$ is the
midpoint between $\lambda/(2|y|)$ and $1$).
Note that since $c < 1$, we must have $F(c) <1$. Furthermore,
since $1/N \to 0$ as $N \to \infty$ we must eventually have
$\lambda (2|y| - 2b/N)^{-1} + b/N < c$. The
result follows.
\end{proof}

\begin{proof}[Proof of Lemma~\ref{lemma:1DDelta}]
If the entries of $|A|$ are i.i.d.~with distribution with density
function $f$ and cumulative distribution $F$, the formula for
order statistics (see~\cite[Equation~(2.3.1)]{DaNa:03} which
we use with $s=n$ and $r=n-1$) gives
\begin{equation}\label{eq:deltaDensityExp}
    \probab(\delta < t) =
    N(N-1)\int_{0}^{\infty} (F(u))^{N-2} f(u)  (F(u+t) - F(u)) \rmd u. 
\end{equation}
			
We will use this formula to study the case where $A$ has exponential,
uniform and Gaussian entries respectively. Starting with the
exponential distribution, we observe that $A=|A|$ in this case and
therefore we have $F(u) = 1 - \rme^{-u}$ and
$f(u) =\rme^{-u}$. Integrating~\eqref{eq:deltaDensityExp}
by parts yields
\begin{equation*} 
    \probab(\delta < t) = N\lim_{u \to \infty} (F(u))^{N-1} [F(u+t) - F(u)] - N
    \int_{0}^{\infty}(F(u))^{N-1} [f(u+t) -f(u)] \rmd u
\end{equation*}
and since $F(u+t) - F(u) = \rme^{-u} (1-\rme^{-t})$, which tends to~0 when
$u\to\infty$, 
and $f(u+t) - f(u)= -\rme^{-u}(1-\rme^{-t}) = -f(u) (1-\rme^{-t})$ we obtain
\begin{eqnarray*} 
   \probab(\delta < t) &=& N (1-\rme^{-t})\int_{0}^{\infty}(F(u))^{N-1}
   f(u) \rmd u\\
   &=& (1-\rme^{-t})\lim_{u \to \infty} [(F(u))^{N}  - (F(0))^N] = 1-\rme^{-t}.
\end{eqnarray*}
			
For the uniform distribution, assume that $N$ is large
enough so that $t/N < 1$. Integrating by parts~\eqref{eq:deltaDensityExp}
again we obtain
\begin{align*}
    \probab(\delta < t/N)
    =& N\lim_{u \uparrow 1} F(u)^{N-1} [F(u+t/N) - F(u)]\\
    & - N \int_{0}^{1} F(u)^{N-1}[f(u+t/N) - f(u)] \rmd u\\
    = &- N \int_{0}^{1} F(u)^{N-1}[f(u+t/N) - f(u)] \rmd u =
    N\int_{1-t/N}^1 F(u)^{N-1} f(u) \rmd u \\
    =& 1 - (F(1-t/N))^N
\end{align*}
where the second to last equality follows because $f(u+t/N) = f(u)$
provided $u$ is in $(0,1-t/N)$ and for $u$ larger than $1-t/N$ we have
$f(u+t/N) = 0$. Furthermore, for the uniform distribution
$F(1-t/N) = 1-t/N$ and thus
$\lim_{N \to \infty} F(1-t/N) = \lim_{N \to \infty}(1-t/N)^N = \rme^{-t}$.
			
The process for the Gaussian is more involved.  Without loss of generality, we can assume that $\sigma =1$: other $\sigma$ reduce to this case by considering $\delta/\sigma$. Let us calculate
$\displaystyle\lim_{N \to \infty}\probab(\delta \sqrt{2\ln(N)} < t)$.
From \eqref{eq:deltaDensityExp}, we must have
\begin{equation} \label{eq:deltaDensityNorm}
	\probab(\delta \sqrt{2\ln(N)} < t)
	= N(N-1)\int_{0}^{\infty} (F(u))^{N-2} f(u)
	\int_{u}^{u+\frac{t}{\sqrt{2\ln(N)}}} f(v) \rmd v \rmd u 
\end{equation}
where $f\left(u \right)=\sqrt{\frac{2}{\pi}}e^{-\frac{\left(u^2+\mu^2 \right)}{2}}\cosh(\mu u)$ for $u \geq 0$ is the density function
for the absolute value of a normal random variable. By the mean
value theorem, for each $u$
we have
\[
\int_{u}^{u+\frac{t}{\sqrt{2\ln(N)}}} f(v) \rmd v
= \frac{tf(c_u)}{\sqrt{2\ln(N)}}
\]
for some $c_u \in (u,u+\frac{t}{\sqrt{2\ln(N)}})$. Hence, integrating by
parts in~\eqref{eq:deltaDensityNorm} and noting that
$\lim_{u \to \infty} f(c_u) =0$ yields
\begin{align*} 
	\probab(\delta \sqrt{2\ln(N)} < t)
	&= N\left[\lim_{u \to \infty}
	(F(u))^{N-1}f(c_u) - (F(0))^{N-1}f(c_0)\right] \\
	&- N\int_{0}^{\infty} (F(u))^{N-1}
	\left(f\left(u+\frac{t}{\sqrt{2\ln(N)}}\right)-f(u)\right)\rmd u\\
	&=- N\int_{0}^{\infty} (F(u))^{N-1}
	\left(f\left(u+\frac{t}{\sqrt{2\ln(N)}}\right)-f(u)\right)\rmd u.
\end{align*}
Because of the form of $f$ and using $\cosh((u+t\alpha) \mu) = \cosh(u\mu)\cosh(t\alpha\mu) + \sinh(u\mu)\sinh(t\alpha\mu)$ we can write $f\left(u+\frac{t}{\sqrt{2\ln(N)}}\right)$ as
\begin{align*}
	&\sqrt{\frac{2}{\pi}} \rme^{-(u^2+\mu^2)/2} \rme^{\frac{-tu}{\sqrt{2\ln(N)}}}
	\rme^{\frac{-t^2}{4\ln(N)}}\cosh(\mu u )\left(  \cosh\left(\frac{t\mu}{\sqrt{2\ln(N)}}\right) +\tanh(\mu u ) \sinh\left(\frac{t\mu}{\sqrt{2\ln(N)}}\right) \right)\\
	&=f(u) \rme^{\frac{-tu}{\sqrt{2\ln(N)}}}
	\rme^{\frac{-t^2}{4\ln(N)}}\left(  \cosh\left(\frac{t\mu}{\sqrt{2\ln(N)}}\right) +\tanh(\mu u ) \sinh\left(\frac{t\mu}{\sqrt{2\ln(N)}}\right) \right) = f(u) g(u)
\end{align*}
where 
\[
g(u) := \rme^{\frac{-tu}{\sqrt{2\ln(N)}}}
\rme^{\frac{-t^2}{4\ln(N)}}\left(  \cosh\left(\frac{t\mu}{\sqrt{2\ln(N)}}\right) +\tanh(\mu u ) \sinh\left(\frac{t\mu}{\sqrt{2\ln(N)}}\right) \right)
\]
and thus
\[
g'(u) = \frac{-tg(u)}{\sqrt{2\ln(N)}} +\rme^{\frac{-tu}{\sqrt{2\ln(N)}}}
\rme^{\frac{-t^2}{4\ln(N)}}\left(\mu \sech^2(\mu u ) \sinh\left(\frac{t\mu}{\sqrt{2\ln(N)}}\right) \right) = \frac{-tg(u)}{\sqrt{2\ln(N)}}  + c_N h(u)
\]
with
\[
c_N = 
\rme^{\frac{-t^2}{4\ln(N)}}\left(\mu  \sinh\left(\frac{t\mu}{\sqrt{2\ln(N)}}\right) \right)\quad  h(u) = \rme^{\frac{-tu}{\sqrt{2\ln(N)}}}\sech^2(\mu u )
\]
Therefore, 
\begin{align*}
	\probab(\delta \sqrt{2\ln(N)} < t)
	&= - N\int_{0}^{\infty} (F(u))^{N-1} f(u) \left(g(u) - 1\right)  \rmd u \\
	&= \lim_{u \to \infty}-(F(u))^{N}
	\left(g(u) - 1\right)  \\
	& +\int_{0}^{\infty} \frac{-tg(u)}{\sqrt{2\ln(N)}}
	(F(u))^{N}  \rmd u + c_N \int_{0}^{\infty} h(u) (F(u))^N \rmd u
\end{align*}
where the last equality holds by a further integration by parts.
Note that as $u \to \infty$, $F(u) \to 1$ and
$g(u) \to 0$. Thus
\begin{align*}
	\probab(\delta \sqrt{2\ln(N)} < t)
	&= 1  +\int_{0}^{\infty} \frac{-tg(u)}{\sqrt{2\ln(N)}}
	(F(u))^{N}  \rmd u + c_N \int_{0}^{\infty} h(u) (F(u))^N \rmd u\\
\end{align*}
For the second of these two integrals, note $c_N \to 0$ as $N \to \infty$ and so
\[
\left|c_N\int_{0}^{\infty} h(u) (F(u))^N \rmd u\right| \leq  c_N\int_{0}^{\infty} \sech^2(\mu u) \rmd u = c_N/\mu \to 0.
\]
For the first integral
$
	\int_{0}^{\infty} \frac{tg(u)}{\sqrt{2\ln(N)}}
	(F(u))^{N}  \rmd u = I_1 + I_2
$ where (using the change of variables $v = tu/\sqrt{2\ln(N)}$)
\begin{align*}
	I_1&:= \bfC_N\int_0^{\infty} \rme^{-v}
	\left(F\left(\frac{v\sqrt{2\ln(N)}}{t}\right)\right)^{N} \rmd v
	\\
	I_2&:=\bfS_N\int_0^{\infty}   \rme^{-v}
	\tanh\left(\frac{\mu v\sqrt{2\ln(N)}}{t}\right) \left(F\left(\frac{v\sqrt{2\ln(N)}}{t}\right)\right)^{N}  \rmd v\\
	\bfC_N& := \rme^{\frac{-t^2}{4\ln(N)}}\left(  \cosh\left(\frac{t\mu}{\sqrt{2\ln(N)}}\right)\right), \quad \bfS_N: = \rme^{\frac{-t^2}{4\ln(N)}}\left(  \sinh\left(\frac{t\mu}{\sqrt{2\ln(N)}}\right)\right)\\
\end{align*}
Now we note that since $F$ is the cumulative density function of the absolute value
of the normal distribution we have $F(x) = [\erf((x+\mu)/\sqrt{2}) + \erf((x-\mu)/\sqrt{2})]/2$. Thus we can use (e.g. \cite[Inequality 7.1.13]{AbSt:64}) to obtain, for any $x > \mu$, that $F(x) \in [1-h(x,8/\pi),1-h(x,4)]$, with
\newcommand{\partnorm}[1]{\frac{\rme^{-(\frac{x#1\mu}{2})^2}}{x#1\mu + \sqrt{(x#1\mu)^2+y}}}
\begin{equation*}
	h(x,y) = \sqrt{\frac{2}{\pi}} \left[\partnorm{+} + \partnorm{-}\right]
\end{equation*}
In particular, for $x > 0$ large enough and $y$ fixed it is easy to see that there exist positive constants $C_1$,$C_2$ so that
$C_2\rme^{-\frac{x^2}{2}}\rme^{x|\mu|}/x \geq h(x,y) \geq C_1 \rme^{-\frac{x^2}{2}}/x$
and therefore when $N$ is sufficiently large
\begin{equation*}
	\left(1 - \frac{C_2tN^{-v^2/t^2}\rme^{\frac{v\sqrt{2\ln(N)}|\mu|}{t}}}{v\sqrt{2\ln(N)}}\right)^N\leq
	\left(F(v\sqrt{2\ln(N)}/t)\right)^N \leq
	\left(1 - \frac{C_1t N^{-v^2/t^2}}{v\sqrt{2\ln(N)}}\right)^N.
\end{equation*}
Hence, pointwise in $v$, 
\[
\lim_{N \to \infty} \left(F(v\sqrt{2\ln(N)}/t)\right)^{N} =
\begin{cases}
	1 & \mbox{if $v^2/t^2 >1$}\\
	0 & \mbox{if $v^2/t^2 < 1$}.
\end{cases}
\]
Thus the integrands of $I_1$ and $I_2$ defined for non-negative $v$
converge pointwise as $N \to \infty$ to the function $g(v)= \rme^{-v}$
for $v > t$ and $g(v) = 0$ otherwise. Furthermore
the integrands of $I_1$ and $I_2$ are dominated (uniformly in $N$) by the integrable function
$\rme^{-v}$ Thus we can apply the dominated
convergence theorem to see that $I_1/C_N$, $I_2/S_N$ both tend to $\rme^{-t}$ as $N \to \infty$. Since $C_N \to 1$ and $S_N \to 0$ as $N \to \infty$, we conclude that
\begin{equation*}
	\lim_{N \to \infty} \probab(\delta \sqrt{2\ln(N)}< t)
	=1 -\rme^{-t}.
\end{equation*}
\end{proof}

\section{Proof of Theorem~\ref{thm:NormalCondWainwright}}

To prove Theorem~\ref{thm:NormalCondWainwright} our strategy
is to use Proposition~\ref{proposition:rhofssigmalb} to
get a bound on $\xul(y,A)$ that depends on the
random quantities $\alpha := \triple{y,A}_2$ and
$\sigma:=\sigma(y,A)$. The following two lemmas yield probabilistic
bounds on $\alpha$ and $\sigma$ for the randomly generated data
$(y,A)$ according to the setup described by {\rm (Si--iii)}. We assume throughout the next two statements that (a0)-(aii) are satisfied. 

\begin{lemma}\label{lemma:AlphaProbabilisticBound}
Set
$\hat{\alpha} = \max\Big\{1,\frac52\sqrt{m}(\eta +
C_{\max}\|v\|_2),
\|\Sigma\|_2^{1/2}(3\sqrt{m}+6\sqrt{N})\Big\}$. Then
$\probab(\alpha \geq \hat{\alpha}) \leq 3\rme^{-m}$.
\end{lemma}

\begin{lemma}\label{lemma:SigmaProbabilisticBound}
There exists universal constants $c_4,c_5$ and $c_6$
such that the following is true.
For any $t>0$ such that
\[
  t<\min\left\{C_{\min}^2(\sqrt{s}
+ \sqrt{m})^{-4},\frac{\gamma\lambda}{2},
{\displaystyle\min_{j\in S}} |v_j| - g(\lambda)\right\}
\]
we have
\begin{align}
  \probab(\sigma \leq t)
  \leq &\ 4 \rme^{-c_6 \min\{m\epsilon^2,s\}} +
  2(N-s) \rme^{-\frac{(\gamma-2t/\lambda)^2}
  {2\rho_u \bar{M}_{N}(\epsilon)}}
  +c_4\rme^{-c_5 \min\{s,\ln(N-s)\}}]\notag\\
  & + 2\rme^{-\frac{(\sqrt{C_{\min}} -
  (\sqrt{s}+\sqrt{m})t^{\frac14})^2}{2t^{\frac12}}}\label{eq:ProbSigmaBound}
\end{align}
where
\begin{equation}\label{eq:MnDefinition}
  \bar{M}_{N}(\epsilon):= \frac{(1+\epsilon)}{C_{\min}\theta_um}
  \left(s\theta_u + \frac{m}{2\ln(N)\phi_N} \right).
\end{equation}
In particular, there exists constants $c_7,c_8$ such that
\begin{equation}\label{eq:SigmaProbabilisticSigmaHatBound}
  \probab(\sigma \leq \hat{\sigma}) \leq c_7 \rme^{-c_8\min\{s,\ln(N-s)\}}
\end{equation}
where $\hat{\sigma}= \min\left\{\frac{C^2_{\min}}{4(\sqrt{s}+\sqrt{m})^4},
\frac{\gamma\lambda}{4},{\displaystyle \min_{j \in S}}
|v_j|- g(\lambda)\right\}$.
\end{lemma}

Assuming for now Lemmas~\ref{lemma:AlphaProbabilisticBound}
and~\ref{lemma:SigmaProbabilisticBound}, the proof of
Theorem~\ref{thm:NormalCondWainwright} is simple.

\begin{proof}[Proof of Theorem~\ref{thm:NormalCondWainwright}]
Recall the function $q(\nu,\xi)$ defined
in~\eqref{eq:q} (and~I.7.1).  
Note that for any fixed $\nu \geq 1$ the
function $f(\xi)= \xi^2/q(\nu,\xi)$ is
increasing on $[0,\infty)$. Indeed, it
suffices to show that $g(\xi):=f(\xi^2)$
is increasing on $[0,\infty)$ and this can be done by noting
that $g(\xi)$ takes the form
$g(\xi)= \xi^4/(C+D\xi + E\xi^2)$ for
positive $C,D,E$ and thus, for $\xi\ge 0$,
\begin{equation*}
  g'(\xi) = \frac{4\xi^3(C+D\xi +E\xi^2)
  - \xi^4(D+2E\xi) }{(C+D\xi+ E\xi^2)^2}
  \geq \frac{3D\xi^4+2E\xi^5}{(C+D\xi+E\xi^2)^2}
  \geq 0.
\end{equation*}
It is also clear that if we instead fix $\xi$ then
$\xi^2/q(\nu,\xi)$ is decreasing in
$\nu$. Similarly, the function
$\sqrt{\xi}/{6\nu}$ is increasing in
$\xi$ and decreasing in $\nu$ for $\nu \in [1,\infty)$.
		
Thus if we take $\sigma=\sigma(y,A) \geq \hat{\sigma}$ and
$1\leq  \alpha = \triple{y,A}_2 < \hat{\alpha}$,
we have (by Proposition~\ref{proposition:rhofssigmalb})
\[
  \xul(y,A) \geq (mN)^{-\frac{1}{2}}\min\left(\frac{\sigma^2}{q(\alpha,\sigma)},
  \frac{\sqrt{\sigma}}{6\alpha},\alpha \right)
  \geq (mN)^{-\frac{1}{2}}\min
     \left(\frac{\hat{\sigma}^2}{q(\hat{\alpha},\hat{\sigma})},
     \frac{\sqrt{\hat{\sigma}}}{6\hat{\alpha}},1 \right) =\hat{\xi}
\]
where we set $\hat{\xi} = \widehat{K}^{-1}$.
		
We conclude that for $\xul(y,A) \leq  \hat{\xi}$ we must have either
$\sigma < \hat{\sigma}$ or $\alpha > \hat{\alpha}$. Therefore
$\probab(\condfsul(y,A) \geq \widehat{K} )
= \probab(\xul(y,A) \leq \hat{\xi}) \leq
\probab(\sigma < \hat{\sigma} \cup \alpha > \hat{\alpha})
\leq \probab(\sigma < \hat{\sigma}) + \probab(\alpha > \hat{\alpha})$.
By Lemmas~\ref{lemma:AlphaProbabilisticBound}
and~\ref{lemma:SigmaProbabilisticBound}, this is bounded above by
$3\rme^{-m} + c_7 \rme^{-c_8\min\{s,\ln(N-s)\}}
\leq \bfc_1 \rme^{-\bfc_2\min\{s,\ln(N-s)\}}$ for some
universal constants $\bfc_1,\bfc_2$, the last inequality by (ai),
completing the proof.
\end{proof}

The remainder of this section is dedicated to proving
Lemmas~\ref{lemma:AlphaProbabilisticBound}
and~\ref{lemma:SigmaProbabilisticBound}. Since $\alpha = \max\{\|y\|_2,\|A\|_2,1\}$, we devise probabilistic bounds
for $\|y\|_2$ and $\|A\|_2$ separately.

\begin{lemma}\label{lemma:NormAProbabilisticBound}
We have $ \probab\big(\|A\|_2 \geq
  \|\Sigma\|_2^{\frac12}(3\sqrt{m}+6\sqrt{N})\big) \leq  \rme^{-m}.
$
\end{lemma}

\begin{proof}
Write $A = U \sqrt{\Sigma}$, where \raisebox{1mm}{\mbox{$\sqrt{~}$}}
represents the matrix square root.
Then, the entries of $U\in\real^{m\times N}$ are i.i.d.~with distribution
$\mcN(0,1)$. Furthermore, for any $t>0$,
\begin{equation*}
\begin{split}
  \probab\big(\|A\|_2 \geq \|\Sigma\|_2^{\frac12}(\sqrt{m}(1+t)+6\sqrt{N})\big)
  &\leq \probab\big(\|U\|_2 \geq \sqrt{m}(1+t)+6\sqrt{N}\big)\\
  &\le \rme^{-\frac{2(\sqrt{m}(1+t))^2}{\pi^2}}\; =\; \rme^{-m\frac{2(1+t)^2}{\pi^2}}
\end{split}
\end{equation*}
the second line by~\cite[(4.6) and Lemma~4.14]{Condition}. 
The result follows by setting $t = 2$.
\end{proof}

\begin{lemma}\label{lemma:NormyProbabilisticBound}$
 \probab\left(\|y\|_2 \geq \frac{5}{2}\sqrt{m}(\eta+C_{\max}\|v\|_2)\right)
 \leq 2\rme^{-m}.$
\end{lemma}

\begin{proof}
As $y=Av+w$, for all $t > 0$, we have
\begin{align*}
   \probab(\|y\|_2 \geq t) &\leq
   \probab\left(\|Av\|_2 \geq
   \frac{t\sqrt{v\transp \Sigma v}}{\eta + \sqrt{v\transp \Sigma v}}\right)
   + \probab\left(\|w\|_2
   \geq \frac{t\eta}{\eta + \sqrt{v\transp \Sigma v}}\right)\\
   &= \probab\left(\frac{\|Av\|_2}{\sqrt{v\transp \Sigma v}} \geq
   \frac{t}{\eta + \sqrt{v\transp \Sigma v}}\right)
   + \probab\left(\frac{\|w\|_2}{\eta} \geq
   \frac{t}{\eta + \sqrt{v\transp \Sigma v}}\right).
\end{align*} 
We next observe that $\|w\|_2/\eta$ and $\|Av\|_2/\sqrt{v\transp  \Sigma v}$
are independent
random variables, both Chi distributed with parameter $m$
(this is obvious for $\|w\|_2/\eta$ whereas for
$\|Av\|_2/\sqrt{v\transp  \Sigma v}$, note that for each
$i=1,2,\dotsc,m$ the random variable $\Sigma_{j=1}^{N} A_{ij} v_j$ is normally
distributed with variance $v\transp \Sigma v$
and the claim follows by the independence of the rows of $A$).
Let $\zeta$ denote any of these two random variables.
By~\cite[Corollary~4.6]{Condition},  
$\probab(\zeta \geq \sqrt{m}+u) \leq \rme^{-\frac{u^2}{2}}$. Hence, taking
$t=\frac{5}{2} \sqrt{m}(\eta+C_{\max}\|v\|_2)$ and noting that
$\sqrt{v\transp \Sigma v} \leq C_{\max}\|v\|_2$ we get
\begin{equation*}
   \probab(\|y\|_2 \geq t) \leq
   2\probab\left(\zeta \geq
   \frac{\frac52\sqrt{m}(\eta + C_{\max}\|v\|_2)}
   {(\eta + \sqrt{v\transp \Sigma v})}\right)
   \leq 2\probab\left(\zeta \geq \frac52\sqrt{m}\right)
   \leq 2\rme^{-\frac98m} \leq 2\rme^{-m}.
\end{equation*}
\end{proof}
We can now prove Lemma~\ref{lemma:AlphaProbabilisticBound}.
\begin{proof}[Proof of Lemma~\ref{lemma:AlphaProbabilisticBound}]
By the definition of $\alpha$, 
\begin{align*}
  \probab(\alpha  \geq \hat{\alpha})
  &\leq \probab(\|y\|_2 \geq \hat{\alpha})
  + \probab(\|A\|_2 \geq \hat{\alpha}) \\
  & \le \probab\Big(\|y\|_2 \geq \frac52\sqrt{m}(\eta + C_{\max}\|v\|_2)\Big)
  + \probab\big(\|A\|_2 \geq  \|\Sigma\|_2^{\frac12}(3\sqrt{m}+6\sqrt{N})\big)\\
  &\le\; 3\rme^{-m}
\end{align*}
where the last inequality holds by Lemmas~\ref{lemma:NormAProbabilisticBound} 
and~\ref{lemma:NormyProbabilisticBound}.
\end{proof}

The proof of Lemma~\ref{lemma:SigmaProbabilisticBound} relies on
three results from Wainwright~\cite{Wainwright:09}:
Theorem~\ref{thm:WainwrightMainTheorem} and
the following two lemmas.

\begin{lemma}\emph{(\cite[p. 2193]{Wainwright:09})}
Under the setup in Section \ref{section:SetupForNormDist} and assumptions
{\rm (a0)} and {\rm (ai)}, there exists a universal constant $c_0$ such that,
for any $t \leq \gamma$ we have
\begin{equation*}
  \probab(\max_{j \in S^{\mathsf{c}}} |Z_j| \geq 1-t)
  \leq 4 \rme^{-c_0 \min\{m\epsilon^2,s\}} + 2(N-s)
  \rme^{-\frac{(\gamma-t)^2}{2\rho_u \bar{M}_{N}(\epsilon)}}
\end{equation*}
where $\bar{M}_N(\epsilon)$ is defined as in \eqref{eq:MnDefinition},
$Z_j$ is defined by 
\begin{equation}\label{eq:ZjRDefinition}
  Z_j = A^*_j \left[A_S (A^*_S A_S)^{-1} \sg(\hat{z}) +
    R\left(\frac{2w}{\lambda}\right)\right],
   \quad R = I_m - A_S(A^*_SA_S)^{-1}A^*_S
\end{equation}
and where $\hat{z}$ solves a \emph{restricted lasso problem}, specifically,
$\hat{z}$ is a solution of
\begin{equation}\label{def:RestrictedLasso}
   \argmin_{\hat{x} \in \real^{|S|}}  \|A_S\hat{x}-y\|^2_2 + \lambda \|\hat{x}\|_1.
\end{equation} 
\eproof
\label{lemma:WainwrightSigma1}
\end{lemma}

Note that by (UL4) and (UL5), the definition of $Z_j$ is independent of
the choice of $\hat{z}$ in the lemma above. 

\begin{lemma}{\bf \cite[Lemma~9]{Wainwright:09}}
Let $X$ be an $m \times s$ matrix such that each row is an i.i.d~random
column vector with distribution $\mcN(0,W)$ for some covariance matrix $W$.
Then, for all $\tau >0$
\[
  \probab\left(\left\|\left(\frac{X^*X}{m}\right)^{-1} - W^{-1}\right\|_2
  \geq \frac{u(m,s,\tau)}{C_{\min}}\right) \leq 2\rme^{-\frac{m\tau^2}{2}}
\]
where $u(m,s,\tau) = 2(\sqrt{s/m}+\tau) + (\sqrt{s/m}+\tau)^2$.\eproof
\label{lemma:WainwrightSigma2}
\end{lemma}

\begin{proof}[Proof of Lemma~\ref{lemma:SigmaProbabilisticBound}]
We define $E$ to be the event that $\sigma(y,A) \leq t$ and $E_1$ to
be the event that the conclusion of Theorem~\ref{thm:WainwrightMainTheorem}
does not hold. 
Note that the matrix $A^*_SA_S$ is invertible with probability $1$, so we
assume throughout that it is indeed invertible.
		
On the event $E_1^{\mathsf{c}}$, there is a unique minimiser $x$ in $\mUL(y,A)$. Furthermore, $x$ has support $S$ and
$\sg(x_S) = \sg(v_S)$.
Thus if $E_1^{\mathsf{c}}$ occurs, 
\begin{enumerate}
\item $\sigma_1(y,A)= \lambda/2 - \|A_{S^{\mathsf{c}}}^*(Ax-y)\|_{\infty}$,
\item $\sigma_2(y,A):=   (\|(A^*_SA_S)^{-1}\|_2)^{-1}$,
\item $\sigma_3(y,A):= \min_{i \in S} |x_i|$.
\end{enumerate}
		
Our proof starts from the observation that 
\begin{equation} \label{eq:ProbEDecomposition}
\mathbb{P}(E) \leq \mathbb{P}(E \cup E_1)
=  \mathbb{P}(E_1) + \mathbb{P}[E_1^{\mathsf{c}}
  \cap (\sigma_1 \leq t)]
  + \mathbb{P}[ E_1^{\mathsf{c}}
  \cap (\sigma_2 \leq \sqrt{t})]
  + \mathbb{P} [E_1^{\mathsf{c}}  \cap (\sigma_3 \leq t)]
\end{equation}
together with the fact that Theorem~\ref{thm:WainwrightMainTheorem} 
gives us a bound for $\mathbb{P}(E_1)$.  
We must thus analyse the events $E_1^{\mathsf{c}} \cap (\sigma_1 \leq t)$,
$E_1^{\mathsf{c}} \cap (\sigma_2 \leq \sqrt{t})$
and $E_1^{\mathsf{c}}  \cap (\sigma_3 \leq t)$.
		
Starting with $E_1^{\mathsf{c}} \cap (\sigma_1 \leq t)$,
if $E_1^{\mathsf{c}}$ occurs then the only solution to the
lasso problem is the vector $x$ with
support exactly $S$. Thus the solution to the
restricted lasso problem~\eqref{def:RestrictedLasso}
is given by $x_S$. 
		
By the KKT conditions for~\eqref{def:RestrictedLasso}
at $x_S$, we have 
$A^*_S (A_S x_S - A_Sv_S - w) = A^*_S (A_S x_S - y) = -\lambda \sg(x_S)/2$
so that
$x_S - v_S = (A^*_SA_S)^{-1}[A^*_Sw - \lambda\sg(x_S)/2]$.
Therefore for $j \in S^{\mathsf{c}}$,
\begin{align*}
	A^*_j(Ax-y) &= A^*_j(A_S(x_S-v_S)-w)\\
	&= A^*_j(A_S(A^*_SA_S)^{-1}[A^*_Sw - \lambda \sg(x_S)/2]-w) \\
        &= -\frac{\lambda A^*_j\left[R\left(\frac{2w}{\lambda}\right)
        + A_S(A^*_SA_S)^{-1}\sg(x_S)\right]}{2} = -\frac{\lambda Z_j}{2}
\end{align*}
where $Z_j$ is defined as in \eqref{eq:ZjRDefinition}. 
		
Thus if $\sigma_1 \leq t$ then we must have
$\lambda/2 -t \leq \|A^*_{S^{\mathsf{c}}}(Ax-y)\|_{\infty}
= \lambda \max\limits_{j \in S^{\mathsf{c}}} |Z_j|/2$
and it follows that for any $\gamma \geq t$,
\begin{align}
\probab(\sigma_1 \leq t)
&\leq \probab\left(\max_{j \in S^{\mathsf{c}}} |Z_j|
> 1-\frac{2t}{\lambda}\right)\notag \\
&\leq 4 \rme^{-c_0 \min\{m\epsilon^2,s\}}
+ 2(N-s) \rme^{-\frac{(\gamma-2t/\lambda)^2}{2\rho_u \bar{M}_{N}(\epsilon)}}
\label{eq:ProbSigma1Fail}
\end{align}
the last inequality from Lemma~\ref{lemma:WainwrightSigma1}.
		
Next, let us analyse $\sigma_2$ assuming that the event $E_1^{\mathsf{c}}$ occurs.
If $\sigma_2 \leq \sqrt{t}$ then since $x$ is supported on $S$, we must have
$\|(A^*_SA_S)^{-1}\|_2 \geq 1/\sqrt{t}$. In particular, if this occurs then 
\begin{equation*}
\left \|\frac{(A^*_SA_S)}{m}^{-1} - (\Sigma_{SS})^{-1}\right\|_2
\geq (\sqrt{t} m)^{-1} - \|(\Sigma_{SS})^{-1}\|_2
= \frac{1}{m\sqrt{t}}-\frac{1}{C_{\min}}
= \frac{C_{\min} - m\sqrt{t}}{C_{\min}\, m\sqrt{t}}
\end{equation*}
		
Because $t < C_{\min}^2(\sqrt{s} + \sqrt{m})^{-4}$, some simple algebra gives
$
 \tau:= \sqrt{\frac{C_{\min}}{m\sqrt{t}}} - \sqrt{\frac{s}{m}} -1 > 0
$
and moreover if $u(m,s,\tau) = 2(\sqrt{s/m}+\tau) + (\sqrt{s/m}+\tau)^2$
then some more simple algebra yields 
$
  u(m,s,\tau) + 1
  = \frac{C_{\min}}{m\sqrt{t}}.
$
Thus 
\begin{align}
\probab (\|(A^*_SA_S)^{-1}\|_2 \geq 1/\sqrt{t})
&\leq \probab\left(\left \|\left(\frac{A^*_SA_S}{m}\right)^{-1}
 - (\Sigma_{SS})^{-1}\right\|_2 \geq \frac{C_{\min} - m\sqrt{t}}{C_{\min}\,
 m\sqrt{t}}\right)\notag \\
&= \probab\left(\left \|\left(\frac{A^*_SA_S}{m}\right)^{-1}
 - (\Sigma_{SS})^{-1}\right\|_2
\geq \frac{u(m,s,\tau)}{C_{\min}}\right)\notag\\
&\leq 2\rme^{-\frac{\left(\sqrt{C_{\min}} - (\sqrt{s}+\sqrt{m})t^{1/4}\right)^2}{2t^{1/2}}}
\label{eq:ProbSigma2Fail}
\end{align}
where the final inequality follows from Lemma~\ref{lemma:WainwrightSigma2}.
		
We can immediately see that $\sigma_3(y,A) \leq t$ and $E_1^{\mathsf{c}}$
cannot occur simultaneously: if $E_1^{\mathsf{c}}$ occurs then for $j \in S$
we have 
$|v_j| - |x_j| \leq |v_j - x_j| \leq g(\lambda)$ and so
$\sigma_3 (y,A) = {\displaystyle\min_{j\in S}} |x_j|
\geq {\displaystyle{\min_{j \in S}}}|v_j| - g(\lambda)> t$.
We conclude that
\begin{equation}
 \probab [E_1^{\mathsf{c}}  \cap (\sigma_3 \leq t)] = 0.
 \label{eq:ProbSigma3Fail}
\end{equation}
Equation \eqref{eq:ProbSigmaBound} follows from~\eqref{eq:ProbEDecomposition},
Theorem~\ref{thm:WainwrightMainTheorem}, \eqref{eq:ProbSigma1Fail},
\eqref{eq:ProbSigma2Fail} and~\eqref{eq:ProbSigma3Fail}. 
		
All that remains is to prove \eqref{eq:SigmaProbabilisticSigmaHatBound}.
We first claim that with the specific choice of $\hat{\sigma}$ we have
\begin{equation}\label{eq:ComplicatedSigmaHat}
  (N-s)\rme^{-\frac{(\gamma - 2\hat{\sigma}/\lambda)^2}
  {2\rho_u \bar{M}_N(\epsilon)}}\leq \rme^{\frac{-\ln(N-s)}{2}}.
\end{equation}
Indeed, since $\hat{\sigma} < \gamma \lambda/4$ we have
$-(\gamma - 2\hat{\sigma}/\lambda)^2/(2\rho_u \bar{M}_N(\epsilon))
\leq -\gamma^2/(8\rho_u \bar{M}_N(\epsilon))$.
Using the definition of $\rho_u$ gives
$\gamma^2/(8\rho_u \bar{M}_N(\epsilon)) =
(8C_{\min}\theta_u \bar{M}_N(\epsilon))^{-1}$. Also,
equation~\eqref{eq:normalMRequiredMeasurements} implies
\begin{equation}\label{eq:m1}	
\frac{m}{1+\epsilon} >
12s\ln(N-s)\theta_u + \frac{6m\ln(N-s)}{\ln(N)\phi_N}
\end{equation}
and hence,
\begin{eqnarray*}
\frac{1}{8C_{\min}\theta_u \bar{M}_N(\epsilon)}
&\underset{\eqref{eq:MnDefinition}}{=} &
\frac{m}{8(1+\epsilon)}
\left(s\theta_u + \frac{m}{2\ln(N)\phi_N}\right)^{-1}\\
&\underset{\eqref{eq:m1}}{>} &
\ln (N-s) \left(\frac18\left(12s\theta_u + \frac{6m}{\ln(N)\phi_N}\right)\right)
\left(s\theta_u + \frac{m}{2\ln(N)\phi_N}\right)^{-1}\\
&=& \frac32\ln(N-s).
\end{eqnarray*}
It follows that 
$$
\ln(N-s) -\frac{1}{8C_{\min}\theta_u \bar{M}_N(\epsilon)}
< \ln(N-s)-\frac32\ln(N-s) = -\frac{\ln(N-s)}{2}
$$
and thus
\begin{equation*}
\rme^{\ln(N-s)-\frac{(\gamma - 2\hat{\sigma}/\lambda)^2}
{2\rho_u \bar{M}_N(\epsilon)}}
\leq \rme^{\ln(N-s) -\frac{1}{8C_{\min}\theta_u \bar{M}_N(\epsilon)}}
\leq \rme^{\frac{-\ln(N-s)}{2}}
\end{equation*}
proving the claim. 
		
We further
claim that
\begin{equation*}
2\rme^{-\frac{\left(\sqrt{C_{\min}} - (\sqrt{s}+\sqrt{m})t^{\frac14}\right)^2}{t^{\frac12}}}
\end{equation*}
is increasing in $t$ when
$(\sqrt{s}+\sqrt{m})t^{\frac14} \leq \sqrt{C_{\min}}$.
Indeed, the function
$f(t^4) :=  -(\sqrt{C_{\min}} - (\sqrt{s}+\sqrt{m})t)^2/t^{2}
= - (\sqrt{C_{\min}}/t - (\sqrt{s}+\sqrt{m}))^2$ is increasing in $t$
provided that
$t < \sqrt{C_{\min}}/(\sqrt{s}+\sqrt{m})$.
With this result and using that
$\hat{\sigma}\le \frac{C_{\min}^2}{4(\sqrt{s}+\sqrt{m})^4}$
(see Lemma~\ref{lemma:SigmaProbabilisticBound}), we see that
\begin{align}\label{eq:CMin-DeltaBound}
2\rme^{-\frac{\left(\sqrt{C_{\min}} - (\sqrt{s}+\sqrt{m})\hat{\sigma}^{\frac14}\right)^2}
{\hat{\sigma}^{\frac12}}}
&\leq  2\rme^{-2(\sqrt{s}+\sqrt{m})^2\frac{\left[\sqrt{C_{\min}}
- \sqrt{C_{\min}}/(4^{\frac14})\right]^2}{C_{\min}}} \\
&= 2\rme^{-2(1-4^{-1/4})^2(\sqrt{s}+\sqrt{m})^2}
\leq 2\rme^{-s/3}\nonumber
\end{align}
since $2(1-4^{-1/4})^2(\sqrt{s}+\sqrt{m})^2 \geq (\sqrt2-1)^2s \geq s/3$.

We can now conclude \eqref{eq:SigmaProbabilisticSigmaHatBound} in
the following way, starting from \eqref{eq:ProbSigmaBound},
\begin{eqnarray*}
\probab(\sigma \leq t)
&\leq&  4 \rme^{-c_6 \min\{m\epsilon^2,s\}} +
2(N-s) \rme^{-\frac{(\gamma-2t/\lambda)^2}
{2\rho_u \bar{M}_{N}(\epsilon)}}
+c_4\rme^{-c_5 \min\{s,\ln(N-s)\}}
+ 2\rme^{-\frac{\big(\sqrt{C_{\min}} - 
(\sqrt{s}+\sqrt{m})t^{\frac14}\big)^2}{t^{\frac12}}}\\
&\underset{\eqref{eq:ComplicatedSigmaHat}}\leq &
4 \rme^{-c_6 \min\{m\epsilon^2,s\}} + 2\rme^{-\ln(N-s)/2}
+ c_4\rme^{-c_5 \min\{s,\ln(N-s)\}}
+ 2\rme^{-\frac{\big(\sqrt{C_{\min}} - 
(\sqrt{s}+\sqrt{m})t^{\frac14}\big)^2}{t^{\frac12}}}\\
&\underset{\eqref{eq:CMin-DeltaBound}}
\leq& 4 \rme^{-c_6 \min\{m\epsilon^2,s\}} + 2\rme^{-\ln(N-s)/2}
+ c_4\rme^{-c_5 \min\{s,\ln(N-s)\}} +  2\rme^{-s/3} \\
&\underset{{\mathrm{(ai)}}}{\leq} &\max\{4,c_4\}
\rme^{-\min\{c_6,1/2,c_5,1/6\}\min\{s,\ln(N-s)\}}
\end{eqnarray*}
and thus we have shown both \eqref{eq:ProbSigmaBound}
and~\eqref{eq:SigmaProbabilisticSigmaHatBound}, completing the proof.
\end{proof}

\section{Proof of Theorem~\ref{cor:simple}}

\begin{lemma}\label{lem:normal}
Let $w\in\real^q$ be random with i.i.d.~components distributed
as $\mcN(0,1)$. Then, for all $t>0$, 
$
  \probab(\|w\|_\infty \le t) \ge
  1-\frac{2q}{t\sqrt{2\pi}}\rme^{-\frac{t^2}{2}}.
$
\end{lemma}

\begin{proof}
We have $\probab(|w_i|\ge t)\le \frac{2}{t\sqrt{2\pi}}
\rme^{-\frac{t^2}{2}}$ by~\cite[Lemma~2.16]{Condition}. Hence,
$$
 \probab(\|w\|_\infty \ge t) =
 \probab(\vee_{i\le q} |w_i| \ge t) \le
 \frac{2q}{t\sqrt{2\pi}}\rme^{-\frac{t^2}{2}}
$$
and the result follows.
\end{proof}

\begin{proof}[Proof of Corollary~\ref{cor:simple}]
We will use the notations in Theorem~\ref{thm:NormalCondWainwright}
and Corollary~\ref{cor:isotropic}. Also, we take
$\bar{c}:=\max\{c_3,1\}$ where $c_3$ is the universal constant
in (aii').

Because $m,s\le N/9<N/8$ we have
$\frac{1}{4(\sqrt{s}+\sqrt{m})^4}> \frac{1}{N^2}$. Also,
by hypothesis (ii) and since $\eta=0$,
${\displaystyle \min_{j\in S}}|v_j|-g(\lambda)\ge \frac{1}{N^2}$. Along with
hypothesis (iii) it follows that
\begin{equation}\label{eq:bound-hat-sigma}
\hat{\sigma}\ge \frac{1}{N^2}.
\end{equation}
Similarly, using that $m\le N/9$, 
\begin{equation}\label{eq:bound-hat-alpha}
  \hat{\alpha}\le \max\left\{\sqrt{2N}\|v\|_2,\sqrt{N}+6\sqrt{N}\right\}
 \le 7\,\trip{v}_2\sqrt{N}.
\end{equation}
Finally, because 
$\lambda\le\frac{2N\trip{v}_2}{\bar{c}}$
(by (ii)) and $\bar{c}\ge1$, $\lambda\sqrt{N}\le 2N^{1.5}\trip{v}_2$
and, hence,
\begin{equation}\label{eq:bound-lambda}
  (1+\lambda\sqrt{N})\le 3N^{1.5}\trip{v}_2.
\end{equation}
It follows that
\begin{eqnarray*}
 \frac{q(\hat{\alpha},\hat{\sigma})}{\hat{\sigma}^2}
 &=& \frac{96\hat{\alpha}^5}{\hat{\sigma}^2}
 +\frac{12\hat{\alpha}^3(1+\lambda \sqrt{N})}{\hat{\sigma}^{1.5}}
	+ \left(\frac{2\hat{\alpha}^3}{\lambda}
        + 3\hat{\alpha}\right)\frac{1}{\hat{\sigma}}\\
&\underset{\eqref{eq:bound-lambda},{\mathrm{(iii)}}}{\le}&
\frac{96\hat{\alpha}^5}{\hat{\sigma}^2}
 +\frac{36\hat{\alpha}^3 N^{1.5}\trip{v}_2}{\hat{\sigma}^{1.5}}
	+ \frac{\hat{\alpha}^3N^2
        + 3\hat{\alpha}}{\hat{\sigma}}\\
&\underset{\eqref{eq:bound-hat-alpha}}{\le}&
\frac{96\cdot7^5\cdot N^{2.5}\trip{v}_2^5}{\hat{\sigma}^2}
 +\frac{36\cdot 7^3\cdot N^{3}\trip{v}_2^4}{\hat{\sigma}^{1.5}}
	+ \frac{7^3 N^{3.5}\trip{v}_2^3
        + 21\,N^{0.5}\trip{v}_2}{\hat{\sigma}}\\
&\underset{\eqref{eq:bound-hat-sigma}}{\le}&     
  96\cdot7^5\cdot N^{6.5}\trip{v}_2^5
 +36\cdot7^3\cdot N^{6}\trip{v}_2^4
 + 7^3 N^{5.5}\trip{v}_2^3
 + 21\,N^{2.5}\trip{v}_2\\
&\le& 1626184\, N^{6.5}\trip{v}_2^5.
\end{eqnarray*}
Also,$
 \frac{6\hat{\alpha}}{\sqrt{\hat{\sigma}}}
 \leq 42 N^{1.5}\trip{v}_2.
$
We conclude that 
\begin{equation}\label{eq:boundK}
  \hat{K}\le (mN)^\frac12
  \max\left\{\frac{q(\hat{\alpha},\hat{\sigma})}{\hat{\sigma}^2},
     \frac{6\hat{\alpha}}{\sqrt{\hat{\sigma}}}\right\}
     \le \frac{1626184}{3} N^{7.5} \trip{v}_2^5 < 542062 N^{7.5}\trip{v}_2^5. 
\end{equation}

We next note that hypotheses (i) and (ii) in our
statement imply that (ai') and (aii') are satisfied. Hence,
Corollary~\ref{cor:isotropic} (with $\eta=0$) may be applied to
deduce that
$$
 \probab(\condfsul(b,U) \geq \widehat{K} )
 \leq \bfc_1 \rme^{-\bfc_2\min\{\ln(N-s),s\}} \leq \bfc_1\rme^{-\bfc_2\ln(N/2)}
 \leq \bfc_1 \left(\frac{N}{2}\right)^{-\bfc_2}
$$
the second inequality by our hypothesis on $s$.

Because of Lemma~\ref{lem:normal} with $q=mN$ and $t=N^2$, we have
$$
\probab(\|U\|_{\max}\le N^2)\ge
 1-\frac{2mN}{N^2\sqrt{2\pi}}\rme^{-\frac{t^2}{2}}
 \ge 1-\rme^{-\frac{N^4}{2}}.
$$
Also, $\|b\|_\infty\leq \|U\|_{2\infty}\|v\|_2 \leq \sqrt{N}\|U\|_{\max}\|v\|_2$.
Therefore, $\triple{b,U}_{\max}\leq N^{2.5}\trip{v}_2$ and
 \begin{align}
\lambda + \lambda^{-1} &\leq \frac{2N\|v\|_2 }{\bar{c}} + N^2/2 \leq 2N^2\trip{v}_2 \label{eq:corllinvBound},\\
\nTrOneStar{b,U} &\leq mN \nTrMax{b,U} \leq N^{4.5} \trip{v}_2\label{eq:corntronestarbound}, 
\end{align}
with probability at least $1-\rme^{-\frac{N^4}{2}}$. 

Using~\eqref{eq:boundK} and Theorem~I.1.2, 
we conclude that the cost of the algorithm
for random $(b,U)$ satisfying our assumptions is bounded by
$\Oh\big(N^3(\log_2 N^{14.5} \trip{v}_2^7)^2\big) $
and using ~\eqref{eq:boundK},~\eqref{eq:corllinvBound} and~\eqref{eq:corntronestarbound}, the maximum number of digits accessed by the algorithm is bounded by
\[
\Oh(\left\lceil \log_2\left(\max\{\lambda + \lambda^{-1},N,\nTrOneStar{b,U},
\condfsul(b,U)\}\right)\right\rceil) = \Oh(\log_2(N\trip{v}_2))
\]	
with probability at least
$1-(\bfc_1 \left(\frac{N}{2}\right)^{-\bfc_2}
+\rme^{-\frac{N^4}{2}})\ge 1-\bfC_1 N^{-\bfC_2}$ for some appropriately
chosen constants $\bfC_1,\bfC_2$. Because the hypotheses of
Theorem~\ref{thm:WainwrightMainTheorem} hold it follows that
with probability greater than
$1-c_1 \rme^{-c_2\min\{\ln(N-s),s\}} \geq 1- c_1 \left(\frac{N}{2}\right)^{-c_2}$ 
we also have that $\supp(x)=\supp(v)$ (here $\mUL(b,U)=\{x\}$).
The result follows by changing, if necessary, the values of $\bfC_1$
and $\bfC_2$.
\end{proof}

\section{Proof of Theorem \ref{thm:failure_bounds}}

\newcommand{\dyadvi}[1]{d_{#1}^{\mathrm{vec}}}
\newcommand{\dyadmi}[1]{d_{#1}^{\mathrm{mat}}}
\newcommand{\dyadvmi}[1]{(\dyadvi{#1},\dyadmi{#1})}

\subsection{Definitions for the computational problem}
We start by recalling some definitions from \cite{BCH1}:
\begin{definition}[The LASSO computational problem]
	\label{definition:LassoComputationalProblem}
	For some set $\Omega \subset \real^{m} \times \real^{m \times N}$,
	which we call the \emph{input} set, the \emph{LASSO
		computational problem on} $\Omega$ is the collection
	$\{\Xi,\Omega, \mathbb{B}^N,\Lambda\}$ where $\Xi:\Omega\to2^{\mathbb{B}^N}$
	is defined
	as in~\eqref{eq:LassoComp} and
	\[
	\Lambda = \{f^{\mathrm{vec}}, f^{\mathrm{mat}}\} \text{ with }
	f^{\mathrm{vec}}:\Omega\to\real^m, f^{\mathrm{mat}}:\Omega\to\real^{m\times N}
	\]
	are defined by $f^{\mathrm{vec}}(y,A) = y$ and
	$f^{\mathrm{mat}}(y,A) = A$ for all $(y,A) \in \Omega$.
\end{definition}
\newcommand{\dyadic}{\mathbf{D}}
\newcommand{\pr}{\mathbb{P}}
We want to generalise the LASSO computational problem so that we work
with \emph{inexact inputs}. To do so, 
we will consider the collection
of all functions $f_{n}^{\mathrm{vec}}: \Omega \to \real^m$
and $f_{n}^{\mathrm{mat}}: \Omega \to \real^{m\times N}$ satisfying
\begin{align}\label{eq:Lambda_y}
	\|f_{n}^{\mathrm{vec}}(y,A) - y\|_{\infty} &\leq 2^{-n}, \quad  \nmax{f_{n}^{\mathrm{mat}}(y,A)-A} \leq 2^{-n} 
\end{align}

\begin{definition}[Inexact LASSO computational problem]
	\label{definition:Omega_tilde_Delta_m}
	The \emph{inexact LASSO computational problem on $\Omega$} (ILCP) is the
	quadruple
	$
	\{\tilde \Xi,\tilde \Omega,\mathbb{B}^N,\tilde \Lambda\},
	$
	where
	\begin{equation}
		\begin{split}
			\tilde \Omega = \big\{ &\tilde \iota =
			\{(f_{n}^{\mathrm{vec}}(\iota),f_{n}^{\mathrm{mat}}(\iota)\}_{n\in\nat}
			\mid \iota = (y,A) \in \Omega \text { and }\\
			&f_{n}^{\mathrm{vec}}: \Omega \to \real^m,
			f_{n}^{\mathrm{mat}}:\Omega\to\real^{m\times N}
			\text{ satisfy~\eqref{eq:Lambda_y} respectively}\big\}
		\end{split}
	\end{equation}
	It follows from~\eqref{eq:Lambda_y} that there is a unique
	$\iota =(y,A)\in\Omega$ for which
	$
	\tilde \iota = \big\{(f_{n}^{\mathrm{vec}}(\iota),
	f_{n}^{\mathrm{mat}}(\iota))\big\}_{n\in\nat}.
$
	We say that this $\iota\in\Omega$ \emph{corresponds} to
	$\tilde\iota\in\tilde\Omega$ and we set
	$\tilde \Xi: \tilde \Omega \rightrightarrows \mathbb{B}^N$ so
	that $\Xi(\tilde \iota) = \Xi(\iota)$, and
	$\tilde \Lambda = \{\tilde f_{n}^{\mathrm{vec}},\tilde
	f_{n}^{\mathrm{mat}}\}_{n \in \nat}$,
	with $\tilde f_{n}^{\mathrm{vec}}(\tilde\iota) = f_{n}(\iota)$,
	$\tilde f_{n}^{\mathrm{mat}}(\tilde\iota) = f_{n}^{\mathrm{mat}}(\iota)$
	where $\iota$ corresponds to $\tilde \iota$.
\end{definition}
\begin{definition}[General Algorithms for the ILCP]
	\label{definition:Algorithm}
	A \emph{general algorithm} for 
	$\{\tilde \Xi,\tilde \Omega,\mathbb{B}^N,\tilde \Lambda\}$,
	is a mapping 
	$\Gamma:\tilde\Omega\to\mathbb{B}^N$
	such that, for every $\tilde \iota\in\tilde \Omega$,
	the following conditions hold:
	\begin{enumerate}[label=(\roman*)]
		\item there exists a nonempty subset of evaluations
		$\Lambda_\Gamma(\tilde \iota) \subset\tilde\Lambda $ with 
		$|\Lambda_\Gamma(\tilde \iota)|<\infty$\label{property:AlgorithmFiniteInput},
		\item the action 
		of $\,\Gamma$ on $\tilde\iota$ is uniquely determined by
		$\{f(\tilde \iota)\}_{f \in \Lambda_\Gamma(\tilde\iota)}$,
		\label{property:AlgorithmDependenceOnInput}
		\item for every $\iota^{\prime} \in\Omega$ such that
		$f(\iota^\prime)=f(\tilde\iota)$
		for all $f\in\Lambda_\Gamma(\tilde\iota)$, it holds that
		$\Lambda_\Gamma(\iota^{\prime})=\Lambda_\Gamma(\tilde\iota)$.
		\label{property:AlgorithmSameInputSameInputTaken}
	\end{enumerate}
\end{definition}
Specific to this paper is the study of \emph{finite precision algorithms} and \emph{algorithms that are correct on all inputs that can be represented exactly}. To define these concepts, we first define the following, which is similar to \cite[Definition 8.23]{opt_big}.
\begin{definition}[Number of correct `digits' for the ILCP]\label{def:no-of-input-bits}
	Given a general algorithm $\Gamma$ for the inexact LASSO computational problem, we define the \emph{`number of  digits' required on the input} according to
	\[
	D_{\Gamma}(\tilde \iota) :=\sup\{m\in\mathbb{N} \, \vert \,  \text{at least one of } f_{m}^{\mathrm{vec}},f_{m}^{\mathrm{mat}} \in  \Lambda_{\Gamma}(\tilde \iota) \}.
	\]
\end{definition}
\begin{definition}[Finite precision algorithm]
	\label{definition:FPAlgorithm}
	A \emph{finite precision general algorithm} with precision $2^{-k}$ for 
	$\{\tilde \Xi,\tilde \Omega,\mathbb{B}^N,\tilde \Lambda\}$,
	is a general algorithm
	$\Gamma:\tilde\Omega\to\mathbb{B}^N$
	such that, for every $\tilde \iota\in\tilde \Omega$,
	the number of correct digits
	$D_{\Gamma}(\tilde \iota) \leq k$.
\end{definition}
\begin{definition}\label{def;correct}
	A general algorithm $\Gamma$ for the ILCP defined on $\Omega \subset \real^m \times \real^{m \times N}$ is said to be \emph{correct on all inputs that can be represented exactly in $\Omega$} if for all $y \in \dyadic^m$ and $A \in \dyadic^{m \times N}$ with $(y,A) \in \Omega$, we have 
	$\Gamma(\tilde \iota) \in \Xi(y,A)$ whenever  \[
	\tilde \iota = ((f_{1}^{\mathrm{vec}}(y,A),f_{1}^{\mathrm{mat}}(y,A))
	(f_{2}^{\mathrm{vec}}(y,A),f_{2}^{\mathrm{mat}}(y,A)),\dotsc )
	\] with  $f_{n}^{\mathrm{vec}}: \Omega \to \real^m$
	and $f_{n}^{\mathrm{mat}}: \Omega \to \real^{m\times N}$ defined for all $n \in \mathbb{N}$ so that
	$
		f_{n}^{\mathrm{vec}}(y,A) =y\quad  f_{n}^{\mathrm{mat}}(y,A)=A .
	$
\end{definition}
\subsection{Proof of Theorem \ref{thm:failure_bounds}}
We begin by stating and proving two preliminary lemmas:
\begin{lemma} \label{lem:existsOpenFail}
	Suppose that the open set $T \in \real^m \times \real^{m \times N}$ contains a point with condition infinity. Then for any $S_1 \subseteq\{1,2,\dotsc,N\}$ there exists a point $P = (y^P,A^p) \in T$ and an $\epsilon > 0$ so that if $\|y - y^P\|_{\infty} \leq \epsilon$, $\|A - A^p\|_{\infty} \leq \epsilon$ then $\Xi(y,A) = \{S_2\}$ with $S_2 \neq S_1$.
\end{lemma}
\begin{proof}
	Since there is a point with condition infinity inside $T$, there must exist some point $y^1,A^1$ so that $\{S_1\} \neq \Xi(y^1,A^1)$ (otherwise $\Xi(b,U) = \{S_1\}$ for all $(b,U) \in T$ and hence $\cond{b,U} < \infty$ since $T$ is open). In particular,  either $\Xi(y^1,A^1) = \{W\} \neq \{S_1\}$ in which case $\mULMS(y^1,A^1) = \Xi(y^1,A^1) = \{W\} \neq \{S_1\}$ or, by Lemma~I.8.2 $|\mUL(y^1,A^1)| \geq 2$ in which case there must exist $W \in \mULMS(y^1,A^1)$ with $W \neq S$. 
	
	From Lemma~I.8.8 if we set $x \in \mUL(y^1,A^1)$ with $\supp(x) = W$ and $E$ to be the diagonal matrix with entries $(\indic_{1 \in W^c}, \indic_{2 \in W^c},\indic_{3 \in W^c},\dotsc\indic_{N \in W^c})$ on the diagonal, then $x = \mUL(y^1,A^1(I-\delta E))$ for $\delta > 0$ sufficiently small. Since $x^1$ is the unique vector in $\Sol(y^1,A^1(I-\delta E))$, we must have $\sigma_2(y^1,A^1(1-\delta E)) > 0$. We also have 
	\begin{align*}
		\|[A^1(I-\delta E)]^*_{S^c} (A^1x - y^1)\|_{\infty} &=  \|(I_{S^c} - \delta E_{S^c})(A^1)^*_{S^c} (A^1x-y^1)\|_{\infty}\\
		&=  \|(1-\delta)I(A^1)^*_{S^c} (A^1x-y^1)\|_{\infty}\\
		& = (1-\delta)\|(A^1)^*_{S^c} (A^1x-y^1)\|_{\infty} \leq (1-\delta) \lambda/2
	\end{align*}
	and thus $\sigma_1(y^1,A^1(I-\delta E)) \geq \delta \lambda/2$. Finally, $\sigma_3(y^1,A^1(I-\delta E)) > 0$ since $\mUL(y^1,A^1(I-\delta E))$ is a singleton. We have thus shown that $\sigma(y^1,A^1(I-\delta E)) > 0$. Thus by Proposition \ref{proposition:rhofssigmalb} we must have $\stabsupp(y^1,A^1(I - \delta E)) > 0$ and thus if we choose $\delta$ sufficiently small and positive so that $P = (y^1,A^1(I-\delta E)) \in S$ then there exists an $\epsilon > 0$ so that if $\|y - y^P\|_{\infty} \leq \epsilon$, $\|A - A^p\|_{\infty} \leq \epsilon$ then $\Sol(y,A) = x^1$ and in particular $\Xi(y,A) = \supp(x) = W \neq S$.
\end{proof}
\begin{lemma}\label{lem:FinitePrecisionSameInputOutput}
	Suppose that an inexact algorithm $\Gamma: \tilde \Omega \to \mathcal{M}$ is a finite precision algorithm with precision $k$. If there are sequences $\{s^1_n\}_{n=1}^{\infty} \in \tilde \Omega$ and $\{s^2_n\}_{n=1}^{\infty} \in \tilde \Omega$ such that $s^1_i = s^2_i$ for $i=1,2,\dotsc,k$, then $\Gamma(\{s^1_n\}_{n=1}^{\infty}) = \Gamma(\{s^2_n\}_{n=1}^{\infty})$.
\end{lemma}
\begin{proof}
	By the definition of a finite precision algorithm, the set $D_{\Gamma}(\{s^1_n\}_{n=1}^{\infty})$ has cardinality at most $k$. In particular, $f(\{s^1_n\}_{n=1}^{\infty}) = f(\{s^2_n\}_{n=1}^{\infty})$ for all $f \in \Lambda_\Gamma(\{s^1_n\}_{n=1}^{\infty})$ and thus $\Lambda_\Gamma(\{s^1_n\}_{n=1}^{\infty}) = \Lambda_\Gamma(\{s^2_n\}_{n=1}^{\infty})$ by Definition \ref{definition:Algorithm} (iii). We conclude that  $\Gamma(\{s^1_n\}_{n=1}^{\infty}) = \Gamma(\{s^2_n\}_{n=1}^{\infty})$ by by Definition \ref{definition:Algorithm} (ii).
\end{proof}

\begin{proof}[Proof of Theorem \ref{thm:failure_bounds}]
	For shorthand, we write $B = \mathcal{B}_{\infty}(b,U,r)$. Since $\cond{b,U} = \alpha$, we have $\stabsupp(b,U) = \alpha^{-1}$. In particular, $\Xi(b,U) = \{S_1\}$ for some set $S_1 \subseteq \{1,2,\dotsc N\}$, otherwise, this would contradict Definition \ref{def:stabSupp}. Furthermore there exists a point $(\hat y,\hat A)$ with $\cond{\hat y, \hat A} = \infty$ and  $(\hat y, \hat A) \in \overline{\mathcal{B}}_{\infty}(b,U,\alpha^{-1}) \subset  B$ since $r > \alpha^{-1}$.
	
	We can thus apply Lemma \ref{lem:existsOpenFail} to the open set $B$ and the support set $S_1$ to obtain a point $(y^P,A^P)$ and an $\epsilon > 0$ so that $\mathcal{B}_{\infty}(y^P,A^P,\epsilon) \subseteq  B$ and if $(y,A) \in \mathcal{B}_{\infty}(y^P,A^P,\epsilon)$ then $S_1 \notin \Xi(y,A) = \{S_2\}$.
	
	We next define three sets
	\begin{gather*}
		F_1:= \{ (y,A) \in B\, \vert \, \Xi(y,A) = \{S_1\} \},\quad F_2:= \{ (y,A) \in B \, \vert \, \Xi(y,A) = \{S_2\} \},\\
		F_3:= \{ (y,A) \in B \, \vert \, |\Xi(y,A)| = 1, \, \Xi(y,A) \neq \{S_1\}, \, \Xi(y,A) \neq \{S_2\} \}.
	\end{gather*}
	and note that the above argument shows that both $F_1$ and $F_2$ each contain an open ball. We now proceed to argue separately for $\Gamma_1$ and $\Gamma_2$
	
	\textit{The result for $\Gamma_1$:}
	We will first construct $\Delta_1$ information for $F:= F_1 \cup F_2 \cup F_3$ so that $\Gamma_1$ fails on this $\Delta_1$ information for $F$. Note that $F$ is the set of points $(y,A)$ in $B$ so that $|\Xi(y,A)| = 1$. We have already noted that $\stabsupp(b,U) = \alpha^{-1}$ implies that $\mathcal{B}_{\infty}(b,U,\alpha^{-1}) \in F_1$. Furthermore, by \cite[Lemma 4]{LASSOUNIQUE}, $B \setminus F$ has measure $0$. so this will suffice to show that $\Gamma_1$ fails on $F$.
	
	We thus define the $\Delta_1$ information requireed. Since $F_1$ and $F_2$ each contain an open set they each must contain at least one point in $\dyadic^m \times \dyadic^{m \times N}$, say, $\dyadvmi{1}$ and $\dyadvmi{2}$ respectively.
	
	Take arbitrary $\hat f_n^{\mathrm{vec}}: B \to \dyadic^m$ and $\hat f_n^{\mathrm{mat}} : B \to \dyadic^{m \times N}$ so that for all $(y,A) \in B$ we have
	\begin{equation*}
		\|\hat f_{n}^{\mathrm{vec}}(y,A) - y\|_{\infty} \leq 2^{-n}, \quad  \nmax{\hat f_{n}^{\mathrm{mat}}(y,A)-A} \leq 2^{-n}
	\end{equation*}
	and $\hat f_{n}^{\mathrm{vec}}(\dyadvmi{i}) = \dyadvi{i} \quad  \hat f_{n}^{\mathrm{mat}}(\dyadvmi{i}) = \dyadmi{i}$ for $i \in \{1,2\}$.
	
	We now define the following $\Delta_1$ information:
	\begin{equation*}
		f_n^{\mathrm{vec}}(\iota) = \begin{cases} 
			\dyadvi{1} & \text{ if } n \leq k\\
			\hat f_n^{\mathrm{vec}}(\iota) & \text{ if } n > k
		\end{cases},\quad f_n^{\mathrm{mat}}(\iota) = \begin{cases} 
			\dyadmi{1} & \text{ if } n \leq k\\
			\hat f_n^{\mathrm{mat}}(\iota) & \text{ if } n > k
		\end{cases} \quad \text{ for } \iota \in F_2 \cup F_3
	\end{equation*}
	and 
	\begin{equation*}
		f_n^{\mathrm{vec}}(\iota) = \begin{cases} 
			\dyadvi{2} & \text{ if } n \leq k\\
			\hat f_n^{\mathrm{vec}}(\iota) & \text{ if } n > k
		\end{cases},\quad f_n^{\mathrm{mat}}(\iota) = \begin{cases} 
			\dyadmi{2} & \text{ if } n \leq k\\
			\hat f_n^{\mathrm{mat}}(\iota) & \text{ if } n > k
		\end{cases} \quad \text{ for } \iota \in F_1
	\end{equation*}
	and $f_n(\iota) = \hat f_n(\iota)$ whenever $\iota \in B \setminus F$.
	
	This defines $\Delta_1$ information for $(y,A) = \iota \in B$: this is clear for $\iota \in B \setminus F$. For $\iota \in F_2 \cup F_3$ this holds because $\|\dyadvi{1} - y\|_{\infty} \leq \|\dyadvi{1} - b\|_{\infty} + \|b - y\|_{\infty} \leq2^{-r+1} \leq 2^{-k} \leq 2^{-n}$ and $\|\dyadmi{1} - A\|_{\max} \leq \nmax{\dyadmi{1} - U} + \|U-A\|_{\infty} \leq 2^{-r+1} \leq 2^{-k}\leq 2^{-n}$  for $n \leq k$ and similarly for $\iota \in F_1$ we have $\|\dyadvi{2} - y\|_{\infty}, \|\dyadmi{2} - A\|_{\infty} \leq 2^{-n}$. Slightly abusing notation, we denote $\Gamma(\iota) = \Gamma(\{f_n^{\mathrm{vec}}(\iota),f_{n}^{\mathrm{mat}}(\iota)\}_{n=1}^{\infty})$.

	With this $\Delta_1$ information, we now use the precondition that $\Gamma$ is correct on all inputs that can be represented exactly in $\Omega$ to see that if $\{g^i_n\}_{n=1}^{\infty}$ is the constant sequence given by $g^i_n= (\dyadvi{i},\dyadmi{i})$ then $\Gamma(\{g^i_n\}_{n=1}^{\infty}) = S_i$. We now apply Lemma \ref{lem:FinitePrecisionSameInputOutput} to obtain $\Gamma(\iota) =  S_1$ when $\iota \in F_2 \cup F_3$ and $\Gamma(\iota) = S_2$ when $\iota \in F_1$. By the definitions of $F_1,F_2, F_3$ and $F$ this proves that $\Gamma(\iota) \notin \Xi(\iota)$ for all $\iota \in F$.

	\textit{The result for $\Gamma_2$:}
	As before, there exists a point $\dyadvmi{1} \in B$ with $\Xi\dyadvmi{1} = S_1 = \Xi(b,U)$.

	Take arbitrary $\hat f_n^{\mathrm{vec}}: B \to \dyadic^m$ and $\hat f_n^{\mathrm{mat}} : B \to \dyadic^{m \times N}$ so that for all $(y,A) \in B$ we have
	\begin{equation*}
		\|\hat f_{n}^{\mathrm{vec}}(y,A) - y\|_{\infty} \leq 2^{-n}, \quad  \nmax{\hat f_{n}^{\mathrm{mat}}(y,A)-A} \leq 2^{-n}
	\end{equation*}
	
	We now define the following $\Delta_1$ information:
	\begin{equation*}
		f_n^{\mathrm{vec}}(\iota) = \begin{cases} 
			\dyadvi{1}& \text{ if } n \leq k\\
			\hat f_n^{\mathrm{vec}}(\iota) & \text{ if } n > k
		\end{cases},\quad f_n^{\mathrm{mat}}(\iota) = \begin{cases} 
			\dyadmi{1} & \text{ if } n \leq k\\
			\hat f_n^{\mathrm{mat}}(\iota) & \text{ if } n > k
		\end{cases}
	\end{equation*}
	The same arguments as in the previous part show that this provides $\Delta_1$ information for every $\iota \in B$. We again slightly abuse notation by setting $\Gamma(\iota) = \Gamma(\{f_n^{\mathrm{vec}}(\iota),f_{n}^{\mathrm{mat}}(\iota)\}_{n=1}^{\infty})$.
	
	By Lemma \ref{lem:FinitePrecisionSameInputOutput}, if we define $\{g^i_n\}_{n=1}^{\infty}$ as before
	we must have that $\Gamma_2(\iota) = \Gamma_2(\{g^1_n\}_{n=1}^{\infty})$ for all $\iota \in B$. Now if $\Gamma_2(\{g^1_n\}_{n=1}^{\infty}) = S_1$ then $\Gamma_2(\iota) \neq \Xi(\iota)$ for all $\iota \in F_2 \cup F_3$. Similarly, if $\Gamma_2(\{g^1_n\}_{n=1}^{\infty}) \neq S_1$ then $\Gamma_2(\iota) \neq \Xi(\iota)$ for all $\iota \in F_1$. Either way, since both $F_1$ and $F_2$ contain an open set, we have shown that $\Gamma_2$ fails on an open set.
\end{proof}

\bibliographystyle{abbrv}
\bibliography{condReferences}

\end{document}